\documentclass [11pt]{article}

\usepackage[
        a4paper,
        left=2.5cm,
        right=2.6cm,
        top=3cm,
        bottom=4cm,
        	vmargin=2cm
]{geometry}

\usepackage{setspace}
\usepackage{enumitem}
\usepackage{amsmath, amssymb, amsfonts, amsthm, tikz}
\usepackage[pdftex,colorlinks,backref=page,citecolor=blue,bookmarks=false]{hyperref}
\usepackage{cleveref}
\usepackage[square,sort,comma,numbers]{natbib}
\usepackage{pifont}
\date{}

\usepackage{caption}
\newtheorem{theorem}{\bf Theorem}[section]

\newtheorem{claim}[theorem]{\bf Claim}
\newtheorem{lemma}[theorem]{\bf Lemma}

\newtheorem{conj}[theorem]{\bf Conjecture}

\newtheorem{corollary}[theorem]{\bf Corollary}
\newtheorem{observation}[theorem]{\bf Observation}
\newtheorem{definition}[theorem]{\bf Definition}
\newtheorem{question}[theorem]{\bf Question}
\newtheorem{notation}[theorem]{\bf Notation}

\newcommand{\HH}{\mathcal H}

\newcommand{\rhat}{\hat{r}}
\newcommand{\eps}{\varepsilon}

\newcommand{\T}{\mathcal{T}}
\newcommand{\G}{\mathcal{G}}

\newcommand{\N}{\mathbb{N}}
\newcommand{\W}{\mathcal{W}}

\newcommand{\sarrow}{\stackrel{s}{\longrightarrow}}

\DeclareMathOperator{\grey}{grey}

\DeclareMathOperator{\level}{level}
\DeclareMathOperator{\type}{type}
\DeclareMathOperator{\types}{Types}
\DeclareMathOperator{\rd}{rd}
\DeclareMathOperator{\id}{id}

\renewcommand{\root}{\mathrm{root}}

\newcommand{\floor}[1]{\lfloor #1 \rfloor}
\renewcommand{\phi}{\varphi}
\renewcommand{\P}{\mathcal P}
\renewcommand{\phi}{\varphi}

\title{Size-Ramsey numbers of tight paths}

\author{
	Shoham Letzter\thanks{
	Department of Mathematics, 
	University College London, 
	Gower Street, London WC1E~6BT, UK. 
	Email: \texttt{s.letzter}@\texttt{ucl.ac.uk}. 
	Research supported by the Royal Society.
	}
	\and
	Alexey Pokrovskiy \thanks{
	Department of Mathematics, 
	University College London, 
	Gower Street, London WC1E~6BT, UK. 
	Email: \texttt{a.pokrovskiy}@\texttt{ucl.ac.uk}.
	}
	\and
	Liana Yepremyan\thanks{Department of Mathematics, Emory University, 
Atlanta, GA, USA. Email: {\tt lyeprem@emory.edu}. Research is supported by the National Science Foundation grant 2247013: Forbidden and Colored Subgraphs. Research in this paper was conducted while the author was a postdoctoral fellow at London School of Economics, funded by Marie Sklodowska Curie Global Fellowship, H2020-MSCA-IF-2018:846304.}} 

\date{}

\begin{document}

\maketitle

\begin{abstract}

	\setlength{\parskip}{\medskipamount}
	\setlength{\parindent}{0pt}
	\noindent
	
	The \emph{$s$-colour size-Ramsey number} of a hypergraph $H$ is the minimum number of edges in a hypergraph $G$ whose every $s$-edge-colouring contains a monochromatic copy of $H$. 
	We show that the $s$-colour size-Ramsey number of the $r$-uniform tight path on $n$ vertices is linear in $n$, for every fixed $r$ and $s$, thereby answering a question of Dudek, La Fleur, Mubayi and R\"odl (2017).
\end{abstract}

\section{Introduction}
\label{sec:intro}

A central problem in extremal combinatorics is to find conditions under which a graph $G$ is guaranteed to contain another graph $H$, and a more ``robust'' variant of this question asks for the existence of a monochromatic copy of $H$ in an edge-coloured graph $G$. Ramsey~\cite{Ramsey1930} initiated an already century-long line of research around such problems, by considering the numbers $R(k,\ell)$, defined to be the minimum number $n$ such that any red/blue edge-colouring of the complete graph $K_n$ contains a red $K_k$ or a blue $K_{\ell}$, and showing that they are finite. Early results of Erd\H{o}s and Szekeres \cite{ES1935} and Erd\H{o}s \cite{erdos1947some} give general upper and lower bounds on these numbers that remained the state-of-the-art for many years, especially for the interesting range $k = \ell$. 
Nevertheless, over the years several improved bounds on $R(k,k)$ were proved~\cite{spencer1977asymptotic,graham1987numbers, Thomason1988, Conlon2009, sah2023diagonal}, culminating with the major breakthrough of Campos, Griffiths, Morris and Sahasrabudhe \cite{CGMS}, who obtained the first significant improvement on the aforementioned early bound for $R(k,k)$, showing that it is at most $(4-c)^k$ for some constant $c > 0$, thereby solving the most famous open problem in the field. 
%
The study of off-diagonal Ramsey numbers, in particular, $R(3,k)$ has been another central topic in the area, that has driven the development of new probabilistic tools and techniques, such as the triangle-free process. We refer the reader to the most recent preprint on this topic by Campos, Jenssen, Michelen and Sahasrabudhe~\cite{CJMS}, who  improved the best-known lower bound on $R(3,k)$, thus narrowing the gap between the upper and lower bounds to a factor $3+o(1)$. 

More generally,  for two graphs or hypergraphs $\G$ and $\HH$ and an integer $s \ge 2$ write $\G \sarrow \HH$ if in every $s$-edge-colouring of $\G$ there is a monochromatic copy of $\HH$.  The \emph{Ramsey number} of a graph/hypergraph $\HH$ is defined as 	
\begin{equation*}
    r_s(\HH) = \min\{|\G|: \G \sarrow \HH\}.
\end{equation*}
    
In this paper we study a variant of classical Ramsey numbers, introduced by Erd\H{o}s, Faudree, Rousseau and Schelp \cite{erdos1978size} in 1978. The $s$-colour \emph{size-Ramsey number} of a hypergraph $\HH$, denoted $\rhat_s(\HH)$, is the minimum number of \emph{edges} in a hypergraph $\G$ satisfying $\G \sarrow \HH$. Namely,
	\begin{equation*}
		\rhat_s(\HH) = \min\{e(\G): \G \sarrow \HH\}.
	\end{equation*}
	When $s = 2$, we often omit the subscript $2$, and refer to the $2$-colour size-Ramsey number of $\HH$ as, simply, the \emph{size-Ramsey number} of $\HH$. It is easy to see that for any graph $H$, $\rhat(H)\leq {r(H) \choose 2}$, and it is known that for $H=K_n$, this inequality is tight~\cite{erdos1978size}, the latter often attributed to Chv\'atal.
    
    One of the earliest results regarding size-Ramsey numbers of graphs, obtained by Beck \cite{beck1983size} in 1983, asserts that, somewhat surprisingly, the size-Ramsey number of a path is linear in its length; more precisely, Beck showed that $\rhat(P_n) \le 900 n$ for large $n$. The problem of determining the size-Ramsey number of a path has been the subject of many papers \cite{bal2019new,beck1983size,bollobas2001random,bollobas1986extremal,dudek2015alternative,dudek2017some,letzter2016path}, and the currently best-known bounds are as follows:
    \begin{equation*}
        (3.75 + o(1))n \le \rhat(P_n) \le 74n,
    \end{equation*}
    where the lower bound is due to Bal and DeBiasio \cite{bal2019new} and the upper bound is by  Dudek and Pra{\l}at \cite{dudek2017some}. One can easily generalise Beck's arguments to the multicolour setting, showing that $\rhat_s(P_n) = O_s(P_n)$. This multicolour variant has received its own fair share of attention \cite{dudek2017some,krivelevich2019long,bal2019new,dudek2018note}; and currently the best-known bounds are as follows, where the lower bound is due to Dudek and Pra{\l}at \cite{dudek2017some} and  the upper  bound is by Krivelevich \cite{krivelevich2019long}.
    \begin{equation*}
        \rhat_s(P_n) = \Omega(s^2 n), \qquad \rhat_s(P_n) = O(s^2 \log s \cdot n).
    \end{equation*}

    In subsequent work, inspired by previous results on paths and trees, Beck~\cite{Beck1990} asked whether every bounded degree graph has size-Ramsey number linear in its number of vertices. The conjecture was proven to be true for trees  by Friedman and Pippenger~\cite{friedman1987expanding} and for cycles by Haxell, Kohayakawa and Luczak~\cite{haxell1995induced}, but eventually turned out to be false in general, as proven by R\"{o}dl and Szem\'eredi~\cite{rodl2000size} who constructed a sequence of graphs $(G_n)$ with $|G_n|=n$, $\Delta(G_n)\leq 3$ and such that $\rhat(G_n)\geq c n (\log{n})^{1/60}$. More recently, by carefully modifying their construction, Tikhomirov~\cite{tikhomirov2024} improved this lower bound to  $cn \exp(c\sqrt{\log{n}})$ (optimal bound for the method). Arguably one of the most interesting questions in this area remains open; a conjecture by R\"{o}dl and Szem\'eredi~\cite{rodl2000size} suggests that Beck's conjecture is false in a strong way, namely that for 
     every $d\geq 3$ there is $\varepsilon=\varepsilon(d)>0$ and a sequence of graphs $(G_n)$ on $n$ vertices and $\Delta(G_n)\leq d$ such that $\rhat(G_n) \geq n^{1+\varepsilon}$. Regarding upper bounds, Kohayakawa, R\"{o}dl, Schacht, and Szemer\'edi~\cite{kohayakawa2011sparse} proved that any $n$-vertex graph $H$ of maximum degree $d$ satisfies $\rhat(H)\leq n^{2-1/d+o(1)}.$ This was recently improved to  $\rhat(H)\leq n^{2-1/(d-1)+o(1)}$ for $d\geq 4$ by Allen and B\"ottcher~\cite{AllenB2022}. Better bounds have been obtained for triangle-free graphs by Conlon and Nenadov (see \cite{nenadov2016thesis}) and for cubic graphs, first by Conlon, Nenadov and Truji\'c~\cite{Conlon2022}, and later by Dragani\'c and Petrova~\cite{draganicpetrova}. A recent work of Dragani\'c, Kaufmann, Munh'a Correia, Petrova, and Steiner~\cite{DKMCPS} improves the bounds coming from~\cite{kohayakawa2011sparse}  in certain cases by considering treewidth as a parameter, along with maximum degree.

    In contrast, the study of size-Ramsey numbers of hypergraphs was initiated only recently (in 2017) by Dudek, La Fleur, Mubayi and R\"odl \cite{dudek2017size}. One of the first problems proposed in their paper is the following  generalisation of Beck's result \cite{beck1983size} regarding size-Ramsey numbers of paths. The \emph{$r$-uniform tight path} on $n$ vertices, denoted $P_n^{(r)}$, is the $r$-uniform hypergraph on vertex set $[n]$ whose edges are all sets of $r$ consecutive elements in $[n]$. Observe that $P_n^{(2)}$ is the path $P_n$ on $n$ vertices. Dudek, La Fleur, Mubayi and R\"odl  \cite{dudek2017size} asked if, similarly to the graph case, $\rhat\big(P_n^{(r)}\big) = O(n)$ for every $r$. This was answered affirmatively for $r = 3$ by Han, Kohayakawa, Letzter, Mota and Parczyk \cite{han2021size}. For $r \ge 4$, the best-known bound prior to our work was $\rhat(P_n^{(r)}) = O\big( (n \log n)^{r/2}\big)$, due to Lu and Wang \cite{lu2018size}. In our main result we settle the problem, and, in fact, prove a stronger statement where the number of colours is arbitrary (rather than $2$).

	\begin{theorem} \label{thm:main-path0}
		Fix integers $r, s \ge 1$. Then $\rhat_s\big(P_n^{(r)}\big)=O(n)$.
	\end{theorem}

     We did not try to optimize the dependency on $r$ and $k$ in $O_{r,k}(n)$ term but recently, new lower bounds on $\rhat_s\big(P_n^{(r)}\big)$ have been obtained by Winter~\cite{winter2023lower} and Bal, DeBiasio and Lo~\cite{bal2024lower}.
     
	Our methods are also powerful enough to show that Ramsey numbers of other hypergraphs are linear; these hypergraphs include powers of tight paths, tight hypergraph trees and their powers, and long subdivisions of hypergraphs of bounded degree. The proofs for these more general results are significantly more technical (though they follow exactly the same proof approach as Theorem~\ref{thm:main-path0}) and can be found in the unpublished manuscript~\cite{letzter2021size}.
    In fact, to the best of our knowledge, all previous results about graphs and hypergraphs with linear size Ramsey numbers are implied by the main result in \cite{letzter2021size}. In particular, an immediate corollary of our main result there is that fixed powers of bounded degree trees have linear multicolour size Ramsey number, a result due to, independently, Kam\v{c}ev, Liebenau, Wood and Yepremyan (for two colours) and Berger, Kohayakawa, Maesaka, Martins, Mendon\c{c}a, Mota and Parczyk \cite{berger2020size}. This, in turn, generalises previous work about powers of paths \cite{clemens2019size,han2020multicolour} that addressed a question of Conlon~\cite{conlon2016}. 
    Another immediate corollary of the main result in \cite{letzter2021size} is that `long subdivisions' of bounded degree graphs have linear multicolour size-Ramsey numbers, a result by Dragani\'c, Krivelevich and Nenadov \cite{draganic2021rolling}, which confirmed a conjecture of Pak \cite{pak2002mixing}.


\section{Proof overview}
	In a graph $G$, denote by $d_G(x,y)$ the distance between $x$ and $y$ in $G$, i.e.\ the length of the shortest path from $x$ to $y$.
	\begin{definition}[Powers]
		For a graph $G$, and a positive integer $k$, let $G^k$ denote the graph formed by connecting every pair of vertices with $d_G(x,y)\leq k$. We define $\mathcal G^k_{(r)}$ to be the $r$-uniform hypergraph on $V(G)$ whose hyperedges are all $r$-sets of vertices that are within distance at most $k$ from each other in $G$ (i.e.\ $E(\mathcal G^k_{(r)})=\{\{v_1, \dots, v_r\}\subseteq V(G): d_{G}(v_i, v_j)\leq k \text{ for all $i,j \in [r]$}\}$). 
	\end{definition}
	Throughout the paper $r$ will be a fixed integer which is at least $3$, so we abbreviate $\G^k_{(r)}$ as $\mathcal G^k$. 
	To give a linear upper bound on $\rhat_s\big(P_n^{(r)}\big)$ we need to come up with a hypergraph $G$ with $O(n)$ edges that satisfies $G\sarrow P_n^{(r)}$. The host graph we use is a bounded degree expanding graph $G$ and show that for some large constant $p$, we have $\mathcal G^p\sarrow  P_n^{(r)}$. This sort of construction dates back to \cite{clemens2019size}, and has been used in extensively since~\cite{berger2020size,kamcev2019size,han2021size,han2020multicolour,kohayakawa2019size}. 
 There is  one main property of $G$ that we use, namely that 	\begin{equation} \label{eqn:expander-ramsey} 		G^c\longrightarrow \Big(\overbrace{P_{n/k}, \dots, P_{n/k}}^t, K_d\Big).	\end{equation}	
 This is proved in Lemma~\ref{Lemma_expander_ramsey_one_Kn}. 	The basic structure of our proof is the following. 
	We consider an $s$-colouring of $\mathcal G^p$, and suppose for contradiction that there are no monochromatic tight paths of length $\Omega(n)$.
	To each vertex $v\in V(G)$ we associate a rooted tree $T(v)$, denoting $\mathcal{T}=\{T(v): v\in V(G)\}$. 
	Then we repeat the following procedure:
	\begin{enumerate}[label = \ding{118}]
		\item \label{itm:step-1}
			We consider a carefully defined auxiliary colouring $\phi$ of the graph $G^{k}$, based on $F$. By \eqref{eqn:expander-ramsey} (with $t$ an appropriate function of $s$ and $r$), there are two possibilities:
			\begin{enumerate}[label = \rm(\alph*)]
				\item \label{itm:step-a}
					If $\phi$ has a monochromatic path of length $\Omega(n)$ in one of the first $t$ colours, then this will produce a  tight path of length $\Omega(n)$ in $\mathcal G^p$ (contradicting our initial assumption).
				\item \label{itm:step-b}
					If instead $\phi$ has many disjoint monochromatic $K_d$'s in the last colour, then we use their presence to deduce that the original colouring had a certain amount of ``structure'' to it. We use this structure to define a new family $F'$, and return to step 1 with this family $F'$.
			\end{enumerate}
	\end{enumerate}
	By repeating this process, we gradually find more and more ``structure'' in our original colouring. Note that the parameter $k$ changes over iterations but this eventually leads to a contradiction, because  we also prove that there is an absolute limit to how much ``structure'' a colouring can have.

	\vspace{.3cm} 

	The big thing missing from the above overview is an indication of what ``structure'' we find in our colourings. The definition of this is new and quite complicated --- in fact, almost all the details of this paper are there to formally define this structure and prove lemmas to work with it. In the remainder of this section we give some informal indication of how things work.

	First, the sets $T(v)$ we associate with each vertex $v$ above are trees in the graph $G^{k}$ for some suitable power $k$.  Given a set $S$ of leaves of any  rooted tree $T$ we define $\type(S)$ to be the isomorphism class of the subtree of $T$ whose leaves are $S$ and the root is the original root of $T$. The auxiliary colouring of $G^k$ in step 1 roughly joins $x,y\in V(G)$ by an edge whenever there exist two $(r-1)$-sets $A,B\subseteq V(T(v))$ of the same type and a short monochromatic tight path $P_{xy}$ between them, and colour $xy$ by $(\tau, c)$ where $\tau$ is the type of $A$ and $c$ is the colour of $P_{xy}$. It can be shown that long paths in this colouring of $G^k$ can be used to build long tight paths in $\mathcal G^p$ (though this is far from obvious, and requires a lot of additional machinery based around concatenating short tight paths).
	Thus, if step \ref{itm:step-a} fails, then we obtain disjoint sets $\{T(v_1), \dots, T(v_d)\}$ with the property that there is no short monochromatic tight path from a subset $A \subseteq V(T(v_i))$ to a subset $B \subseteq V(T(v_j))$, for any $A, B$ of the same type and any $i\neq j$. This allows us to find a new family $F'=\{T'(v): v\in V(G)\}$, where the sets $T'(v)$ are trees which are taller than the trees $T(v)$, and with the following property: ``for any $x,y$ with $d_G(x,y)\leq k$, there is no short tight path between any pair of subsets of $F'(x)$ and $F'(y)$ of the same type''. We call such a colouring ``$k$-disconnected'' --- and this is exactly the ``structure'' we build throughout the proof. One of our lemmas (Lemma~\ref{Lemma_disconnection_bounder}) says that there do not exist $k$-disconnected families of trees of arbitrarily large height --- this is what we meant by ``there is an absolute limit to how much structure  a colouring can have'', and this is what gives us our final contradiction. 

	The structure of the paper is as follows. In the next section we define expanders and prove that they satisfy \eqref{eqn:expander-ramsey}. In Section~\ref{subsec:tree-intro} we introduce tree families $F$, and give a formal definition of what we mean by ``type''. In Section~\ref{Section_cleaning}, we prove some Ramsey-type lemmas about tree families (which are needed for concatenating short monochromatic tight paths later in the proof, and also for proving Lemma~\ref{Lemma_disconnection_bounder}). In Section~\ref{Section_distances} we formally define $k$-disconnectedness, and prove that there do not exist $k$-disconnected families of trees of arbitrary height (Lemma~\ref{Lemma_disconnection_bounder}). In Section~\ref{Section_augmentation}, we define an operation called ``forest augmentation'' which is how the forest $F$ changes between different iterations of \ref{itm:step-1}. In Section~\ref{Section_versatile} we introduce and study things called ``versatile sets of vertices'', which are another tool we need for concatenating short tight paths. In Section~\ref{Section_proof} we prove Theorem~\ref{thm:main-path0}.
 
\section{Ramsey properties of expanders}
	The Ramsey graphs we construct are based around expanders.
	\begin{definition}[Expander]
		A graph $G$ on $n$ vertices is an \emph{$\eps$-expander} if for any two sets of vertices $A,B$ of size at least $\eps n$, there is an edge starting in $A$ and ending $B$.
	\end{definition}
	It is a standard result that there exist arbitrarily large $\eps$-expanders of bounded degree (e.g.\ $G(n,C/n)$ satisfies this after deleting high degree vertices. See~\cite{letzter2021size}, Proposition 3.2).
	\begin{lemma} \label{Lemma_expander_existence}
		For all $\eps>0$ and large $n$, there exist connected $\eps$-expanders $G$ on $n$ vertices with $\Delta(G)\leq \eps^{-2}$.
	\end{lemma} 
	Our goal in this section is to prove that powers of expanders have a certain Ramsey property (Lemma \ref{Lemma_expander_ramsey_MAIN}).
	We will need the following three well-known bipartite Ramsey results. Here the notation $G \longrightarrow (H_1, \dots, H_k)$ means that in every $k$-edge-colouring of $G$ there is a copy of $H_i$ in colour $i$, for some $i \in [k]$.
	\begin{lemma} \label{Lemma_RPnKnn_2_colour}
		$K_{3n,3n}\longrightarrow (P_n, K_{n,n})$.
	\end{lemma} 
	\begin{proof}
		Let $K_{3n,3n}$ be red/blue coloured. Let the parts of $K_{3n,3n}$ be $X,Y$. 
		It is known that one can partition the vertices of a graph into a path $P$ and two disjoint sets of vertices $A,B$ with $|A|=|B|$, such that there are no edges between $A$ and $B$ (see \cite{ben2012long}). Apply this to the subgraph of red edges to get a red path $P$ and corresponding sets $A,B$. If $|P|\geq n$ we are done, so suppose that $|P|<n$. Since $|A|=|B|$, the number of vertices outside $P$ is even and so $|P|$ is even. Since vertices of $P$ alternate between $X$ and $Y$, this tells us that $|V(P)\cap X|=|V(P)\cap Y|\leq n$, or, equivalently, $|A\cap X|+|B\cap X|=|A\cap Y|+|B\cap Y|\geq 2n$. Without loss of generality, $|A\cap X|\geq |B\cap X|$, which gives $|A\cap X|\geq n$. Adding   $|A\cap X|+|A\cap Y|=|A|=|B|=|B\cap X|+|B\cap Y|$ to  $|A\cap X|+|B\cap X|=|A\cap Y|+|B\cap Y|$, we get $|A\cap X|=|B\cap Y|$. Now, we have two sets $A'=A\cap X$ and $B'=B\cap Y$ of size at least $n$ with no red edges between them (since $A'\subseteq A, B'\subseteq B$), and also no non-edges between them (since $A'\subseteq X, B'\subseteq Y$) i.e.\ they form a blue complete bipartite graph $K_{n,n}$.
	\end{proof}

	\begin{lemma}\label{Lemma_RPnKnn_r_colour}
		$K_{3^rn,3^rn}\longrightarrow \Big(\overbrace{P_n,\dots, P_n}^{\text{$r$ times}}, K_{n,n}\Big)$ for every integer $r \ge 1$.
	\end{lemma}
	\begin{proof}
		We proceed by induction on $r$. The initial case $r=1$ is exactly Lemma~\ref{Lemma_RPnKnn_2_colour}. Consider an $(r+1)$-coloured $K_{3^rn,3^rn}$. By Lemma~\ref{Lemma_RPnKnn_2_colour}, either there is a $P_n$ coloured $1$ or a $K_{3^{r-1}n,3^{r-1}n}$ coloured by $\{2, \dots, r+1\}$. In the former case, we are done immediately. In the latter case, we are done by induction.   
	\end{proof}

	We write $K_n^s$ for the complete $s$-partite graph with parts of size $n$.

	\begin{lemma}\label{Lemma_RPnKnm_multipartite}
		$K_{6^{rs}}\rightarrow \Big(\overbrace{P_n,\dots, P_n}^{\text{$r$ times}}, K_{n}^s\Big)$ for every integers $r \ge 1$ and $s \ge 2$.
	\end{lemma}
	\begin{proof}
		Fix $r \ge 1$ and write $a_s = (2 \cdot 3^r)^{s-1}$.
		We will prove that the following holds for $s \ge 2$.
		\begin{equation*}
			K_{a_sn,\,a_sn} \longrightarrow \Big(\overbrace{P_n,\dots, P_n}^{\text{$r$ times}}, K_{n}^s\Big).
		\end{equation*}
		Notice that this will prove the lemma as $6^{rs} \ge a_s$ for every $s \ge 2$.
		We proceed by induction on $s$.
		The initial case $s=2$ follows from Lemma~\ref{Lemma_RPnKnn_r_colour} (because $a_2 = 2 \cdot 3^r$ and so $K_{a_2n}$ contains a copy of $K_{3^rn,3^rn}$). Suppose then that $s\geq 3$, and that the lemma holds for smaller $s$. Consider an $(r+1)$-coloured $K_{a_sn}$, and suppose that there is no monochromatic $P_n$ with a colour in $[r]$. By Lemma~\ref{Lemma_RPnKnn_r_colour}, there is a copy of $K_{a_{s-1}n, a_{s-1}n}$ coloured $r+1$ (because $a_s = 2 \cdot 3^r \cdot a_{s-1}$). Let $A,B$ be the parts of this $K_{a_{s-1}n,a_{s-1}n}$. By induction applied to $K_n[A]$ and $K_n[B]$, the parts $A$ and $B$ each either contain a monochromatic $P_n$ with a colour in $[r]$, in which case we are done, or a $K_{n}^{s-1}$ coloured $r+1$. The edges between these $K_n^{s-1}$'s are all colour $r+1$, giving a $K_{n}^{2s-2}$ in this colour (and this contains a colour $r+1$ copy of $K_{n}^{s}$, since $2s-2\geq s$ for $s\geq 3$).
	\end{proof}

	The following is a basic property of expanders.
	\begin{lemma}\label{Lemma_expander_rainbow_path}
		Let $S_1, \dots, S_r$ be disjoint sets of size at least $r\eps n$ in an $\eps$-expander $G$ on $n$ vertices. Then there is a path $v_1\dots v_r$ in $G$ with $v_i\in S_i$.
	\end{lemma}
	\begin{proof}
		Notice that between any two disjoint sets $S$ and $T$ in an $\eps$-expander there is a matching of size at least $\min\{|S|,|T|\}-\eps n$ (if $M$ is a maximal $S$-$T$ matching, then we must have $|S\setminus V(M)|=|T\setminus V(M)|\leq \eps n$ as otherwise $\eps$-expansion would give another $S$-$T$ edge from $S\setminus V(M)$ to $T\setminus V(M)$ disjoint from $M$). 
		Set $M_1$ to be a maximal $S_1$-$S_2$ matching and write $m = r\eps n$; so $|M_1| \ge m - \eps n$. Build matchings $M_2, \dots, M_{r-1}$ one by one by taking $M_i$ to be a maximal $(S_i \cap V(M_{i-1}))$-$S_{i+1}$ matching. By the first sentence of the proof (applied with $S = S_i \cap V(M_{i-1})$ and $T = S_{i+1}$), we have 
		\begin{equation*}
			|M_i| \ge \min\{|S_i \cap V(M_{i-1})|, |S_{i+1}|\} - \eps n \ge \min\{|M_{i-1}|, m\} - \eps n \ge m - i\eps n,
		\end{equation*}
		where the last inequality follows from iterating the inequality $|M_i| \ge \min\{|M_{i-1}|,m\} - \eps n$ and using $|M_1| \ge m - \eps n$.
		In particular, $|M_{r-1}| \ge m - (r-1)\eps n > 0$, i.e.\ $M_{r-1} \neq \emptyset$. Also, by construction, each edge of $M_i$ touches an edge of $M_{i-1}$, for every $i \in [r-1]$. These two facts together give a sequence $v_1\dots v_r$ with $v_{i-1}v_i\in M_i$. This is thus a path satisfying the lemma. 
	\end{proof}

	We use $N_G(X)$ to denote the neighbourhood of $X$, namely the set of vertices outside of $X$ with an edge in $G$ into $X$, and omit the subscript $G$ when it is clear from context. 
	\begin{lemma}\label{Lemma_expansion_boosting}
		Let $G$ be a $0.01$-expander of order $n$. Then there is subgraph $H$ of order at least $0.97n$ with $H^{k}$ a $\frac1{5\cdot 2^{k/2-2}}$-expander.
	\end{lemma}
	\begin{proof}
		Set $r=\floor{k/2-1}$ and $\eps=\frac1{5\cdot 2^{k/2-2}}\ge\frac1{5\cdot 2^r}$.
		Let $X\subseteq V(G)$ be a maximal subset with $|X|\leq n/4$ and $|N(X)|\leq 2|X|$. Then $|X|\leq n/100$, since otherwise $|X \cup N(X)|\geq 99n/100$ (if not, we would have $|V(G)\setminus (X \cup N(X))|, |X|\geq 0.01n$. Then the definition of $0.01$-expander applies to give an edge from $X$ to $V(G)\setminus (X \cup N(X))$, which is impossible) so $|N(X)|\geq 99n/100 - |X|\geq 74n/100>2|X|$, a contradiction. Set $H:=G \setminus (X\cup N(X))$ (i.e.\ $G$ with the vertices in $X\cup N(X)$ deleted) to get a graph of order at least $n-3|X|\geq 97n/100$. Note that all $Y\subseteq V(H)$ with $|Y|\leq n/5$ have $|N_H(Y)|>2|Y|$ (since otherwise $X\cup Y$ is a set with $|N(X\cup Y)|\leq 2|X\cup Y|$ and $|X\cup Y|\leq n/100+n/5\leq n/4$, contradicting maximality of $X$).

		Now consider $A,B\subseteq V(H^r)=V(H)$ of size $\eps n$. 
		Write $A_i$ for the set of vertices at distance at most $i$ from $A$ in $H$, and define $B_i$ analogously with respect to $B$. 
		Then $|A_i| \geq |N(A_{i-1})| \ge \min\{2|A_{i-1}|,n/5\}$ for all $i$, giving $|A_r|\geq\min(2^r\eps n, n/5) = n/5$. Similarly, $|B_r|\geq n/5$. Using $0.01$-expansion, this gives an edge between $A_r$ and $B_r$ in $G$ (and hence also in $H$) i.e.\ $d_H(A,B)\leq 2r+1$. Thus there is an $A$-$B$ edge in $H^{2r+1} \subseteq H^k$, as required.
	\end{proof}

	The following is a Ramsey property of expanders.
	\begin{lemma}\label{Lemma_expander_ramsey_one_Kn}
		Let $c\gg d$ ($c\geq 10\cdot 6^{d^2}$ works). Let $G$ be a $0.01$-expander  on $n$ vertices. Then  
		\begin{equation*}
			G^c\rightarrow \Big(\overbrace{P_{n/c}, \dots, P_{n/c}}^{\text{$d$ times}}, K_d\Big).
		\end{equation*}
	\end{lemma} 
	\begin{proof}
		Write $\eps = \frac{1}{5 \cdot 2^{c/2-2}}$. Pick $c$ so that $0.97/6^{d^2}\geq \max\{d\eps,1/c\}$. 
		Let $G^c$ be $(d+1)$-coloured.
		Apply Lemma~\ref{Lemma_expansion_boosting} to find a subgraph $H \subseteq G$ of order $0.97n$ with $H^c$  an $\eps$-expander. 
		Think of this as an $(d+2)$-colouring of $K_{0.97n}$ with non-edges of $H^c$ having colour $0$. Apply Lemma~\ref{Lemma_RPnKnm_multipartite} with $r=s=d$, thinking of $\{0, d+1\}$ as a single colour. This either gives a monochromatic path of length $m:=0.97n/{6^{d^2}}$ coloured by $[d]$ or a $\{0,d+1\}$-coloured $K_{m}^d$. In the former case, we are done (since $m\geq n/c$). In the latter case, let $S_1, \dots, S_d$ be the parts of the $\{0,d+1\}$-coloured $K_m^d$. Using $m\geq d\eps n$, apply Lemma~\ref{Lemma_expander_rainbow_path} to get a path $v_1\dots v_d$ in $H$ with $v_i\in S_i$. Then $H^c[\{v_1, \dots, v_d\}]$ is a $K_d$ coloured $d+1$ (it is complete because the vertices are within distance $d\leq c$ in $H$. It is colour $d+1$ because all its edges go between different parts $S_i, S_j$, and the only colour permitted there is $d+1$).
	\end{proof}

	The following is the main result of this section. 
	\begin{lemma}[Ramsey of expander powers]\label{Lemma_expander_ramsey_MAIN}
		Let $c\gg d, \eps^{-1}$, and $d\geq s$. Let $G$ be an $\eps$-expander on $n$ vertices. Then for every $U\subseteq V(G)$ with $|U|\geq n/2$, we have
		\begin{equation*}
			G[U]^c\rightarrow \Big(\overbrace{P_{n/c},\, \dots,\, P_{n/c}}^{\text{$s$ times}}, \overbrace{K_d \, \dot\cup\, \dots\, \dot \cup\, K_d}^{\text{covering at least $(1-200\eps)|U|$ vertices}} \Big).
		\end{equation*}
	\end{lemma} 
	\begin{proof}
		Without loss of generality, $s=d$ (if $s<d$, just think of any $s$-colouring as a $d$-colouring, with some of the colours not appearing). 
		Let $c=c'/(100\eps)$ where $c'$ is the constant from Lemma~\ref{Lemma_expander_ramsey_one_Kn}.
		Consider a maximum collection of pairwise vertex-disjoint $K_d$'s in $U$ coloured $s+1$, which we denote by $K^1, \dots, K^m$. Suppose, towards contradiction, that $m\leq (1-200\eps)|U|/d$, and define $H=G[U]\setminus V(K^1\cup \dots \cup K^m)$ to get an induced subgraph $H$ of size at least $200\eps |U|\geq 100\eps n$, noting that $H$ is a $0.01$-expander. By Lemma~\ref{Lemma_expander_ramsey_one_Kn} we either get a colour $1$ copy of $K_d$ (contradicting the maximality of our collection), or we get a monochromatic path in one of the other colours of order at least $100\eps n/c' = n/c$, as required. 
	\end{proof}

\section{Ordered forests}\label{Section_tree_assignments}
	\def \Fm {F^-}

	\subsection{Basic definitions} \label{subsec:tree-intro}
		In the paper we work exclusively with trees that are both rooted and ordered. 

		Here ``rooted'' means that the tree has a designated vertex called the \emph{root} and all edges are directed away from the root. A \emph{rooted forest} is a forest whose components are rooted trees.
		A \emph{leaf} in a rooted forest is defined as a vertex with out-degree $0$. The \emph{level} of a vertex is the distance it has to the root in its component. 
		The \emph{height} of a forest is the maximum level of a vertex in the forest. We emphasise that the root of a tree of height at least $1$ is never a leaf (even if it has degree $1$), whereas in a height $0$ tree the single vertex it contains is both a root and a leaf. 
		 A forest is \emph{balanced} if all its leaves are on the same level. We remark that, aside from the beginning of Section~\ref{Section_versatile}, all forests everywhere will be balanced. The out-neighbours of a vertex $v$ are called \emph{children of $v$}, and the in-neighbour is called the \emph{parent of $v$}. Vertices to which there is a directed path from $v$ are called \emph{descendants of $v$}, while vertices from which there is a directed path to $v$ are called \emph{ancestors of $v$}. We remark that $v$ is both a descendent and ancestor of itself. We call two vertices \emph{comparable} if one is a descendent of the other. 
		A $d$-ary forest is a rooted forest where all non-leaf vertices have out-degree $d$ (and in-degree $1$, aside from the root which has in-degree $0$). 

		Ordered trees are defined as follows. 

		\begin{definition}[Ordered trees, forests, tree families]\
			An \emph{ordered forest} is a rooted forest $F$ (i.e.\ one in which every tree of $F$ is rooted), together with an ordering of $V(F)$ such that 
			\begin{enumerate}[label = \rm(F\arabic*)]
				\item \label{itm:ordered-a} 
					If $v'$ is a descendant of $v$ then $v<v'$.
				\item \label{itm:ordered-b}
					For every incomparable $u$, $v$ (namely, $u$ is neither an ancestor nor a predecessor of $v$), with $u<v$, all descendants $u'$ of $u$ and $v'$ of $v$ have $u'<v'$.
			\end{enumerate}
		\end{definition}

		Throughout the paper all trees/forests will be ordered and so we sometimes suppress the word ``ordered'' for clarity.
		The \emph{root set} of $F$, denoted $\root(F)$, is the (ordered) set of roots of $F$. We say ``$F$ is rooted on $V$'' if $\root(F)=V$.
		For ordered forests $F,F'$, we write $F' \leq F$ if $F'$ is an ordered subforest of $F$ with the same height and root set. 

		\begin{notation} \label{not:ordered-tree}
			There is a natural way of identifying  height $h$ ordered forests with subsets of $\mathbb{N}^{h}$: we identify vertices on level $i$ with vectors of the form $(x_{0}, \dots, x_i, 0, \dots, 0)$ and do this level by level. Given the roots $r_1<r_1<\dots r_q$, we will label a root $r_j$ with a vector of form $(j, 0, \dots, 0)$. For levels $i\geq 1$, if a vertex at level $i-1$, $v=(x_0, \dots, x_i, 0, \dots, 0)$ has children $v_1, \dots, v_k$ ordered  $v_1< \dots< v_k$, then we identify $v_i$ with $(x_0, \dots, x_i, i, \dots, 0)$ (see \Cref{fig:ordered-tree}). Note that the ordering of the vertices is exactly the lexicographic ordering on corresponding vectors (i.e.\ $(a_1, \dots, a_h)<(b_1, \dots, b_h)$ if the smallest coordinate with $a_t\neq b_t$ has $a_t<b_t$).
		\end{notation}

		\begin{figure}[h]
			\centering
			\includegraphics[scale = 1]{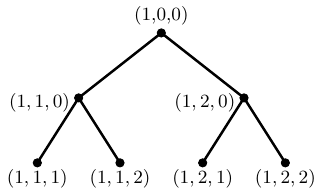}
			\caption{A height 3 ordered tree with the natural labelling in $\N^3$}
			\label{fig:ordered-tree}
		\end{figure}
		\begin{definition} [$L(F)$ and $A^F(S)$]
			Given an ordered forest $F$, we write $L(F)$ for the (ordered) set of leaves of $F$ (i.e.\ vertices with out-degree $0$). 

			Given a set of vertices $S\subseteq V(F)$, we let $A^{F}(S)$ be the induced subforest of $F$ consisting of all ancestors of vertices of $S$ (see \Cref{fig:A-L}).  
		\end{definition}

		\begin{figure}[h]
			\centering
			\includegraphics[scale = .8]{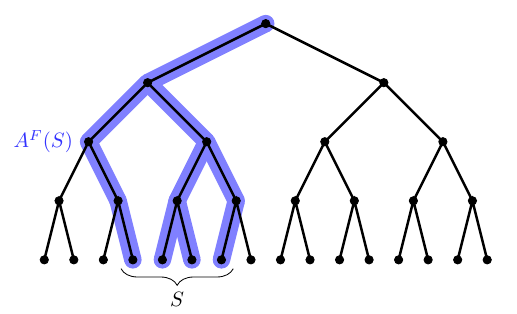}
			\vspace{-.3cm}
			\caption{A set of leaves $S$ and the corresponding $A^F(S)$}
			\label{fig:A-L}
		\end{figure}

		Most of the time it will be clear what $F$ is from context and we abbreviate $A(S)=A^F(S)$.
		It is useful to note that the above two definitions are inverses to each other, in the sense that for any set of leaves $S$ in a forest $F$  we have $L\left(A^F(S)\right)=S$ and for any $F'\leq F$, we have $A^F(L(F'))=F'$ (for this we use that $F$ and $F'$ have the same root set).

		It is also useful to note that $A^F(A^F(S))=A^F(S)$ (since ``is an ancestor of'' is a transitive relation). A consequence of this is that every root of $A^F(S)$ is a root of $F$ (let $x$ be a root of $A^F(S)$. If it is not a root of $F$, then it has an ancestor $y$ in $F\setminus A^F(S)$. But then $y$ is in $A^F(A^F(S))$, but not in $A^F(S)$).

		We use the following notions of homomorphism/monomorphism/isomorphism.

		\begin{definition}[Homomorphisms, monomorphisms and isomorphisms]
			A \emph{homomorphism} $\phi$ of ordered forests, from $F$ to $F'$, denoted $\phi: F \to F'$, is a function from $V(F)$ to $V(F')$, which maps edges to edges and preserves order (i.e.\ for every $u,v\in V(F)$, if $uv\in E(F)$, then  $\phi(u)\phi(v)\in E(F')$, and if $u<v$, then $\phi(u)<\phi(v)$). 
			An injective homomorphism is called a \emph{monomorphism} and a bijective homomorphism is called an \emph{isomorphism}. 
		\end{definition}

		A \emph{copy} of $F$ in $F'$ is an image of $F$ in $F'$ under some monomorphism $\phi : F \to F'$.

		The following two lemmas show that levels/ancestors are preserved by monomorphisms.
		\begin{lemma}\label{Lemma_level_preserved_by_monomorphism}
			Let $F,F'$ be two balanced forests of the same height, and $\phi:F \to F'$ a monomorphism. For any $x\in V(F)$, the vertex $x$ is in the same level in $F$ as $\phi(x)$ is in $F'$.
		\end{lemma}

		\begin{proof}
			Denote the height of $F$ and $F'$ by $h$. Notice that an edge $uv$ in $F$, with $u$ being the parent of $v$, is mapped to an edge $\phi(u)\phi(v)$, with $\phi(u)$ being the parent of $\phi(v)$, due to $\phi$ being order-preserving. Thus, if $v_0 \dots v_h$ is a path in $F$ with $\level(v_i) = i$, then $\phi(v_0) \dots \phi(v_h)$ is a path in $F'$ with $\phi(v_{i-1})$ being the parent of $\phi(v_i)$. Due to $F$ and $F'$ both being height $h$ trees, this implies that $\level(\phi(v_i)) = i = \level(v_i)$. Since every vertex in $F$ is in a path $v_0 \dots v_h$ with $v_i$ having level $i$, it follows that $\level(\phi(v)) = \level(v)$ for every vertex $v$ in $F$.
		\end{proof}

		\begin{lemma}\label{Lemma_A_preserved_by_monomorphisms}
			Let $\phi:F\to F'$ be a monomorphism between two balanced forests of the same height. For every $X\subseteq V(F)$, we have $\phi(A^F (X))=A^{F'}(\phi(X))$.
		\end{lemma}
		\begin{proof}
			We claim that ``$x$ is an ancestor of $y$'' if and only if ``$\phi(x)$ is an ancestor of $\phi(y)$''. Indeed, if $x$ an ancestor of $y$ then there is a path $x = v_0 \dots v_k = y$ where $v_{i-1}$ is the parent of $v_i$, for $i \in [k]$. But then $\phi(x) = \phi(v_0) \dots \phi(v_k) = y$ is a path in $F'$ where $\phi(v_{i-1})$ is the parent of $\phi(v_i)$, showing that $\phi(x)$ is an ancestor of $\phi(y)$. This shows the ``if'' direction. For the ``only if'' part, since by \Cref{Lemma_level_preserved_by_monomorphism} we have $\level(x) = \level(\phi(x))$, the number of ancestors of $x$ in $F$ is the same as the number of ancestors of $\phi(x)$ in $F'$ (and both equal $\level(x) + 1$).
			By definition of $A^F$, we have
			\begin{align*}
				V\left(A^{F'}(\phi(X))\right)
				& = \{y \in V(F') : \text{$y$ is an ancestor of $\phi(x)$ for some $x \in X$}\} \\
				& = \{\phi(y): \text{$y \in V(F)$ and $\phi(y)$ is an ancestor of $\phi(x)$ for some $x \in X$}\} \\
				& = \phi\left(\{y \in V(F): \text{$y$ is an ancestor of some $x \in X$}\}\right) = V\left(\phi\left(A^F(X)\right)\right),
			\end{align*}
			completing the proof.
		\end{proof}

		We can now introduce the crucial definition of the ``type'' of a subset.
		\begin{definition}[Type]
			For a forest $F$ and subset $e \subseteq L(F)$, the \emph{type} of $e$ in $F$, denoted $\type_{F}(e)$, is the isomorphism class of $A^F(e)$ (as an ordered forest). 

			Let $\types(h,r)$ be the set of different types of subsets of size at most $r$ inside ordered trees of height at most $h$. Note that $|\types(h,r)|$ is finite since there are only finitely many rooted trees of height at most $h$ with at most $r$ leaves. 
		\end{definition}

		For an ordered forest $F$, let $F^-$ denote $F$, but with all roots deleted (see \Cref{fig:F-min}). If $\tau$ is the type of $F$, then define $\tau^-$ to be the type of $F^-$ (noting that this depends only on the type $\tau$, and not on $F$).
		\begin{figure}[ht]
			\centering
			\includegraphics[scale = .8]{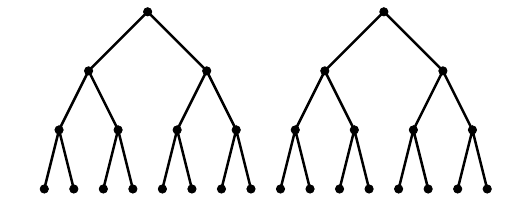}
			\hspace{1cm}
			\includegraphics[scale = .8]{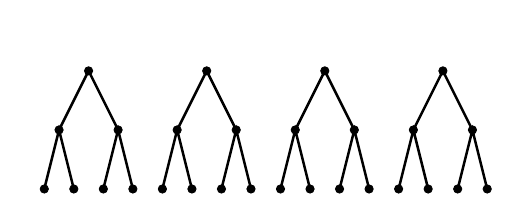}
			\caption{A forest $F$ (on the left) and the forest $F^-$ (on the right)}
			\label{fig:F-min}
		\end{figure}
		
		\begin{lemma}\label{Lemma_A_preserved_by_minus}
			For $X\subseteq L(F)$, we have $A^{F^-}(X)=(A^F(X))^-$.
		\end{lemma}
		\begin{proof}
			By definition,
			\begin{align*}
				V\left(A^{F^-}(X)\right) 
				& = \{y \in V(\Fm): \text{$y$ is an ancestor of some $x \in X$}\} \\
				& = \{y \in V(F) \setminus \root(F): \text{$y$ is an ancestor of some $x \in X$}\} \\
				& = \{y \in V(F): \text{$y$ is an ancestor of some $x \in X$}\} \setminus \root(F) \\
				& = V\left(A^F(X)\right) \setminus \root(F) 
                    = V\left(A^F(X)\right) \setminus \root(A^F(X))
                     = V\left((A^F(X))^-\right). \qedhere
			\end{align*}
		\end{proof}

		In the next two definitions we show how to relate rooted forests to hypergraphs.

		\begin{definition}
			For a set $S$, an \emph{$S$-forest} is a balanced, ordered forest $F$ with $V(F)\subseteq S\times S$ so that all edges are of the form $(y,s)(y',s)$, and all roots are of the form $(s,s)$.
		\end{definition}

		For a set $S$, an element $s \in S$, and an $S$-forest $F$, define $F(s)$ to be the subgraph of $F$ induced on $\{(y,s) \in V(F): y \in S\}$. 

		\begin{observation}
			For a set $S$, an element $s \in S$, and an $S$-forest $F$, the subgraph $F(s)$ is a tree.
		\end{observation}

		\begin{proof}
			For some $(y,s)\in V(F(s))$, let $P$ be the path in $F$ from a root $r$ to $(y,s)$ ($P$ exists since we are in a rooted forest). 
			Write $P$ as $r = (v_1, s_1) \dots (v_k, s_k) = (y,s)$. Then $s_1 = \dots = s_k = s$, because all edges in $F$ are of form $(y,s')(y',s')$, for some $s' \in S$, and $s_k = s$. Thus $r = (s,s)$, because roots have form $(s', s')$, with $s' \in S$.
			This shows that $P$ is a path from $(y,s)$ to $(s,s)$ in $F(s)$. Therefore, $F(s)$ is connected and thus a tree.
		\end{proof}

		We will call the tree $F(s)$ the ``subtree of $F$ rooted at $(s,s)$''.
		We will use $S$-forests as a way to represent a collection of balanced trees of the same height, with roots and vertices in $S$, where each element of $S$ is the root of at most one tree, but can be a vertex in many of these trees.

		For an $S$-forest $F$, a vertex $v\in V(F)$, and $i \le \level(v)$, we let $\pi^F_i(v)$ be the first coordinate of the ancestor of $v$ at level $i$. In particular, $\pi^F_{\level(y,s)}\big((y,s)\big)=y$ and $\pi^F_{0}\big((y,s)\big)=s$. We abbreviate $\pi(v):=\pi^F_{\level(v)}(v)$ and $\pi_0(v):=\pi^F_{0}(v)$ (noting that these do not depend on $F$, since they are just projections on the first/second coordinates of $v$).

		The following is how we relate hypergraphs to forests.
		\begin{definition}[$\HH \otimes_r F$]
			For a hypergraph $\HH$, a $V(\HH)$-forest $F$, and an integer $r \ge 1$, let $\HH\otimes_r F$ be the $r$-uniform hypergraph with vertex set $L(F)$, where an $r$-tuple $e$ is an edge wherever $\pi_0(e)\subseteq f$ for some edge $f$ of $\HH$. Equivalently, $\{(v_1, x_1), \dots, (v_{r}, x_r)\}\in E(\HH\otimes_r F)$ whenever $\{x_1, \dots, x_r\}$ is contained in some edge in $\HH$.
		\end{definition} 

		When $r$ is known from context or its value is not relevant for the arguments, we often abbreviate $\HH \otimes_r F$ to $\HH \otimes F$.

	\subsection{Clean and separated tree assignments}\label{Section_cleaning}
		The goal of this section is to prove that every $V(\HH)$-forest, where $\HH$ is a hypergraph with bounded maximum degree, has a large subforest with the same root set that interacts nicely with $\HH$. 
		\begin{definition}[Clean hypergraphs]
			For a hypergraph $\HH$ and a $V(\mathcal H)$-forest $F$, we say that a colouring of the hypergraph $\HH \otimes F$ is \emph{clean} if any two edges $e,f$ satisfying $\root(e)=\root(f)$ and $\type(e)=\type(f)$ (i.e. $\type\big(e \cap V(F(v))\big) = \type\big(f \cap V(F(v))\big)$ for every $v \in \root(f)$)  have the same colour.
		\end{definition}
		This is equivalent to asking that, for any edge $e\in \mathcal H\otimes F$ and monomorphism $\phi:A(e)\to F$ which preserves roots, $\phi(e)$ has the same colour as $e$.

		\begin{definition}[Separated tree assignments]
			Let $G$ be a graph and $F$  a $V(G)$-forest.
			We say that $F$ is \emph{$d$-separated} on $G$ if for $u,v\in V(G)$ with $d_G(u,v)\leq d$ (namely the distance between $u$ and $v$ on $G$ is at most $d$), we have $\pi(F(u))\cap \pi(F(v))=\emptyset$.
		\end{definition}
		Note that this is equivalent to ``for $u,v\in V(F)$ with $d_G(\pi_0(u), \pi_0(v))\leq d$ we have $\pi(u)\neq \pi(v)$''. Informally, thinking of a $V(G)$-forest as a collection of trees on $V(G)$, where each vertex is the root of at most one tree, this means that trees rooted at vertices that are at distance at most $d$ from each other in $G$ are vertex-disjoint.

		The goal of the section is to prove Lemma~\ref{Lemma_cleaning_forest} which shows that under certain assumptions on $\HH$, and for $d\gg d'$, any forest of $d$-ary trees rooted on $V(\HH)$ contains a $d'$-ary subforest with the same root set that is both clean and separated. We start with establishing ``separated''.
		\begin{lemma}\label{Lemma_separate_2_trees}
			There is a function $D(h,d)$, such that the following holds for $d_1,d_2,d_1',d_2'$ with $d_1 \ge D(h,d_1')$ and $d_2 \ge D(h,d_2')$.
			Let $T_1$ and $T_2$ be balanced, ordered $d_1$- and $d_2$-ary trees of height $h$, whose vertex sets may intersect, but whose roots are distinct. Then there are balanced $d_1'$- and $d_2'$-ary, height $h$ subtrees $T_1'\subseteq T_1$ and $T_2'\subseteq T_2$ with $V(T_1')$, $V(T_2')$ disjoint.
		\end{lemma}

		\begin{proof}
			Assign to each non-root vertex $u$ of $V(T_1) \cup V(T_2)$ a random variable $\tau(u)$ which is chosen to be either $1$ or $2$ uniformly at random, independently of other elements. We claim that, with positive probability, for every non-leaf in $T_i$, at least $d_i/3$ of its children are assigned $i$, for $i \in [2]$. Indeed, the probability that this fails for a particular non-leaf in $T_i$ is at most $e^{-d_i/36}$, by Chernoff's bounds, and thus by a union bound the probability that some non-leaf in $T_i$ has fewer than $d_i/3$ children that were assigned $i$ is at most $d_i^h e^{-d_i/36} < 1/2$ (using $d_i \ge D(h,d_i')$), as claimed. It follows that there exist subtrees $T_i' \subseteq T_i$, for $i \in [2]$, such that $T_i'$ is a $d_i'$-ary tree of height $h$ and its non-root vertices were assigned $i$ (and are not the root of $T_{3-i}$; using $d_i \ge D(h,d_i')$). Then the trees $T_1', T_2'$ are vertex-disjoint (using that the roots of $T_1$ and $T_2$ are distinct), as required.
		\end{proof}

		The above can be used to show that $F$ always contains a separated subforest. In the proof, for a function $f$, we write $f^{i}$ to denote $\overbrace{f \circ \dots \circ f}^{\text{$i$ times}}$. Recall that $F' \le F$ means that $F'$ is a subforest of $F$ with the same root set.
		\begin{lemma}[Separating forests]\label{Lemma_separate_trees}
			Let $d\gg b,h, \Delta$.
			Let $G$ be a graph with maximum degree at most $\Delta$ and $F$ be a $V(G)$-forest of balanced, $d$-ary, height $h$ trees. Then there is a forest $F'\leq F$ of balanced, $b$-ary, height $h$ trees such that $F'$ is $b$-separated on $G$.
		\end{lemma}

		\begin{proof}
			Fix $b,h,\Delta$, and let $f(d'):=D(h, d')$, where $D$ is the function from Lemma~\ref{Lemma_separate_2_trees}. Let $g(d) = \max\{d' : d \ge f(d')\}$, and note that $g(f(d)) \ge d$ for every $d$. 
			Set $d=f^{\Delta^b}(b)$.
			Then $g^i(d) \ge b$ for every $i \le \Delta^b$.

			We modify the forest $F$, by doing the following for each edge $xy$ in $G^b$, one at a time.
			Suppose that $F(x)$ and $F(y)$ are $d_x$- and $d_y$-ary trees of height $h$. Apply \Cref{Lemma_separate_2_trees} to get $g(d_x)$- and $g(d_y)$-ary subtrees $F''(x) \le F(x)$ and $F''(y) \le F(y)$, which are balanced height $h$ forests and are vertex-disjoint, and replace $F(x)$ by $F''(x)$ and $F(y)$ by $F''(y)$.
			Notice that each subtree $F(x)$ is modified at most $\Delta^b$ times, and so it is $g^{i}(d)$-ary for some $i \le \Delta^b$. Since $g^{i}(d) \ge b$, we can choose $F'(x)$ to be a balanced $b$-ary subtree of height $h$ of the tree $F(x)$ at the end of the process.
			Taking $F'$ to be the union of trees $F'(x)$ completes the proof.
		\end{proof}

		 Next we focus on establishing cleanliness.
			\begin{lemma} \label{lem:ramsey-trees}
				There is a function $D(d,r,h,k,s)$ such that the following is true. Let $d,r,h,k,s$ be positive integers and set $D=D(d,r,h,k,s)$.
				Let $F$ be a balanced, $D$-ary ordered forest of height $h$ with $k$ components. Let $S$ be a balanced, ordered forest of height $h$ with $r$  leaves and $k$ components. Given any $s$-colouring of the copies of $S$ in $F$, there is a balanced ordered $d$-ary subforest $F' \le F$ of height $h$ with $k$ components, such that all copies of $S$ in $F'$ have the same colour.
			\end{lemma}

				\begin{proof} 
					We prove the lemma by induction on $h$ and $k$. 
					The initial case is $h=k= 1$, namely $F$ and $S$ are both stars, with $D$ and $r$ leaves, respectively, and so copies of $S$ correspond to $r$-sets of leaves of $F$. This case thus follows from Ramsey's theorem.

					Now fix $(h,k)\neq(1,1)$ and suppose that the lemma holds for all $(h',k')$ with either $h'=h, k'<k$ or with $h'<h, k' \geq 1$. Let $\chi$ be an $s$-colouring of the copies of $S$ in $F$.

					Suppose first that $k = 1$ and $h \ge 2$, so $F$ and $S$ are trees of height $h$.  Denote the root of $F$ by $r_F$ and the root of $S$ by $r_S$.
					Let $F' = F \setminus \{r_F\}$, let $S' = S \setminus \{r_S\}$, and let $\ell$ be the number of components in $S'$; so $1 \le \ell \le r$ as $S$ has $r$ leaves. Consider the $s$-colouring $\chi'$ of copies of $S'$ in $F'$ defined as follows: if $S''$ is a copy of $S'$ in $F'$, consider the copy of $S$ obtained by joining $r_F$ to the roots of the components of $S'$ (since both $S'$ and $F'$ have height $h-1$, the roots of components of $S'$ are roots of components of $F'$, so the copy of $S$ formed in this way is indeed a subtree of $F$), and colour $S''$ by the colour of this copy of $S$ in $F$ according to $\chi$. 

					Let $\alpha$ be such that $d, r, h, s \ll \alpha \ll D$, and set $t = \binom{\alpha}{\ell}$. 
					By choice of parameters, there is a sequence $d_0, \dots, d_t$ such that $d_0 = D$, $d_t = d$, and $r, h, s, d_t \ll d_{t-1} \ll \dots \ll d_0$.
					Let $A$ be a set of $\alpha$ children of $r_F$ in $F$, and let $A_1, \dots, A_{t}$ be any enumeration of its $\ell$-subsets. 
					For $a \in A$, let $T_a$ be the subtree of $F$ rooted at $a$.
					We claim that there exist sequences $T_a = T_a^{(0)} \supseteq T_a^{(1)} \supseteq \dots \supseteq T_a^{(t)}$ for $a \in A$, where $T_a^{(i)}$ is a $d_i$-ary ordered tree of height $h-1$ with the following property: all copies of $S'$ in the forest $\bigcup_{a \in A_i} T_a^{(i)}$ have the same colour. To see this, apply the induction hypothesis with $h' = h-1$ and $k' = \ell$, to $\bigcup_{a \in A_i} T_a^{(i-1)}$, letting $\bigcup_{a \in A_i} T_a^{(i)}$ be the resulting subforest, and take $T_a^{(i)} \subseteq T_a^{(i-1)}$ to be an arbitrary $d_i$-ary ordered subforest of height $h-1$ for $a \in A \setminus A_i$. The subtrees $T_a^{(i)}$ satisfy the requirements. Denote by $c_i$ the colour of any copy of $S'$ in $\bigcup_{a \in A_i} T_a^{(i)}$ (this does not depend on the particular choice of a copy of $S'$).

					Consider the auxiliary colouring of the complete $\ell$-graph on $A$, where the edge $A_i$ has colour $c_i$. Then by Ramsey's theorem and choice of $\alpha$ there is a monochromatic subset $A' \subseteq A$ of size $d$; say the common colour is $c$. Let $F''$ be the $d$-ary tree of height $h$ obtained by reattaching the root of $F$ to the subtrees $T_a^{(t)}$ with $a \in A'$. Then $F'' \subseteq F$ is an ordered $d$-ary tree of height $h$ whose copies of $S$ all have colour $c$, as required.

					Now suppose $k \ge 2$.
					Let $T$ be the first tree in $F$, let $F' = F \setminus V(T)$, let $R$ be the first tree in $S$ and let $S' = S \setminus V(R)$. Choose $d'$ satisfying $d, r, h, k, s \ll d' \ll D$, and let $T'$ be any $d'$-ary subtree of $T$ of height $h$. Enumerate the copies of $R$ in $T'$ by $R_1, \dots, R_t$. By choice of $d'$ and $t$ there exist $d_0, \dots, d_t$ such that $d_0 = D$, $d_t = d$, and $r, s, h, k, d_t \ll d_{t-1} \ll \dots \ll d_0$. 
					By induction, there exist subforests $F' \supseteq F_1 \supseteq \dots \supseteq F_t$, such that $F_i$ is a $d_i$-ary forest of height $h$ with $k-1$ components, whose copies of $S$ in $T' \cup F_i$ that contain $R_i$ all have the same colour, denoted $c_i$.
					Now consider the colouring of copies of $R$ in $T'$, where $R_i$ is coloured $c_i$. By induction, $T'$ contains a $d$-ary subtree $T''$ of height $h$ whose copies of $R$ all have the same colour, say red. Then the forest $T'' \cup F_t$ is a $d$-ary forest of height $h$ with $k$ components whose copies of $S$ are red.
				\end{proof}
			
		The above can be used to show that $\mathcal H\otimes F$ always contains a cleanly coloured hypergraph.	 
		\begin{lemma}[Cleaning forests]\label{Lemma_cleaning_forest}
			Let $d\gg b,h,s, r, \Delta$. Let $\mathcal H$ be an $r$-uniform hypergraph with maximum degree at most $\Delta$ and $F$ a $V(\mathcal H)$-forest of balanced $d$-ary trees of height $h$. 

			Then for any $s$-colouring of $\HH\otimes F$, there is some ordered subforest $F'\leq F$ of balanced, $b$-ary, height $h$ trees such that the subhypergraph $\HH\otimes F'$ is cleanly coloured.
		\end{lemma}
		\begin{proof} 

			Fix $b,h,s,r, \Delta$, and let $f(d)= \max\{D(d,r,h,k,s): k \in [r]\}$, where $D$ is the function from Lemma~\ref{lem:ramsey-trees}.
			Let $g(d) = \max\{d' : d \ge f(d')\}$, noting that $g(f(d)) \ge d$ for every $d \ge 1$.
			Set $\sigma = 2^r|\types(h,r)|$, $\rho = r\Delta+1$, $d = f^{\sigma\rho}(b)$, and observe that $g^{i}(d) \ge b$ for $i \le \rho\sigma$. 

			Let $\HH'$ be the subhypergraph of $\HH$, induced on vertices $v \in V(\HH)$ such that $F(v) \neq \emptyset$.
			We claim that the edges of $\HH'$ can be partitioned into at most $\rho$ matchings (which are collection of pairwise vertex-disjoint edges). Indeed, consider the line graph $L(\HH')$ of $\HH'$, which is the graph on $E(\HH')$ where $ef$ is an edge whenever $e$ and $f$ intersect.
			This graph has maximum degree at most $r\Delta = \rho-1$, and so its chromatic number is at most $\rho$. In other words, there is a proper colouring of $L(\HH')$ with $\rho$ colours, and this corresponds to a partition of $E(\HH')$ into at most $\rho$ matchings. Denote the matchings involved in such a partition by $M_1, \dots, M_\rho$.
			
			Write $F_0 = F$. We define a sequence $F_1, \dots, F_{\rho}$, as follows, so that $F_i \le F_{i-1}$ and $F_i$ is a balanced $g^{\sigma i}(d)$-ary forest of height $h$ for $i \in [k]$. Let $i \in [k]$ and suppose that $F_0, \dots, F_{i-1}$ are defined and satisfy the requirements.
			For each edge $e \in M_i$ we modify the trees $F_{i-1}(v)$ for $v \in e$ (notice that each such tree is a balanced $g^{\sigma (i-1)}(d)$-ary tree of height $h$) in at most $\sigma$ rounds, once for each type $\tau \in \types(h,r)$ and subset $e' \subseteq e$ of size $|\root(\tau)|$. We make sure that after the $j$-th round each $F_{i-1}(v)$ with $v \in e$ is a balanced $g^{\sigma(i-1)+j}(d)$-ary tree of height $h$. Say we are in the $j$-th round, the corresponding type is $\tau$ and the corresponding subset of $e$ is $e'$. Apply \Cref{lem:ramsey-trees} to the ordered forest $\bigcup_{v \in e'}F_{i-1}(v)$ to obtain balanced $g^{(i-1)\sigma+j}(d)$-ary subtrees $F'(v) \le F_{i-1}(v)$ of height $h$, for $v \in e'$, such that all type $\tau$ subforests of $\bigcup_{v \in e'}F'(v)$ have the same colour. Replace $F_{i-1}(v)$ by $F'(v)$ for every $v \in e'$, and replace each $F_{i-1}(v)$ with $v \in e \setminus e'$ by a balanced $g^{(i-1)\sigma+j}(d)$-ary subtree of height $h$, chosen arbitrarily. After these at most $\sigma$ rounds, each $F_{i-1}(v)$ with $v \in e$ is a balanced $g^{k}(d)$-ary tree of height $h$, with $k \le i\sigma$, and for every $\tau \in \types(h,r)$, all type $\tau$ subforests of $\bigcup_{v \in e}F_{i-1}(v)$ with the same root set have the same colour. For each $v \in V(\HH')$, take $F_i(v)$ to be a balanced $g^{i\sigma}(d)$-ary subtree of $F_{i-1}(v)$ of height $h$, chosen arbitrarily, and set $F_i = \bigcup_{v \in V(\HH')}F_i(v)$.

			Notice that $F_\rho$ is a balanced $b$-ary forest (using $b = g^{\sigma\tau}(d)$) of height $h$, satisfying that $F_\rho \le F$ and $\HH \otimes F_\rho$ is cleanly coloured.
		\end{proof}

	\subsection{Distances, short trees, and disconnected tree assignments}\label{Section_distances}
		Given a graph $G$, a $V(G)$-forest $F$, and a hypergraph $\HH$ on $L(F)$ we will need to track what sorts of paths/walks one can find in $\mathcal H$ between different trees of $F$. How long/short a path is will be judged relative to distances in $G$ via the following definition.
		\begin{definition}[$||\ast ||_{F, G}$]
			Let $F$ be a $V(G)$-forest for a graph $G$. For a set $e\subseteq L(F)$, we define
			$||e||_{G}=\max_{x,y\in e}d_{G}(\pi_0(x), \pi_0(y))$. 
		\end{definition}
		We remark that if $P,Q$ have the same root sets, then $||P||_G=||Q||_G$.
		This definition gives a concise way of defining $\mathcal G^t\otimes F$ (recall that $\G^t$ is a shorthand for the $r$-uniform hypergraph $G^t_{(r)}$, which is the $r$-uniform hypergraph on $V(G)$ whose edges are $r$-sets of vertices that are pairwise at distance at most $t$ from each other in $G$).
		\begin{observation}\label{Observation_powers_distance_relation}
			For a graph $G$ and forest $F$ rooted on $V(G)$, the edge set of the hypergraph $\mathcal G^t\otimes_r F$ consists of exactly the $r$-sets $e\subseteq L(F)$ which have $||e||_{G}\leq t$.
		\end{observation}

		If $\HH$ is a hypergraph whose vertex set is ordered, a \emph{tight walk} in $\mathcal H$ is a sequence of vertices $P=v_1 \dots v_k$ so that $\{v_i, v_{i+1}, \dots,$ $v_{i+r-1}\}\in E(\mathcal H)$ for each $i \in [k-r+1]$, and also $v_1<\dots< v_{r-1}$ and $v_{k-r+1}<\dots< v_{k}$. We call $\{v_1, \dots, v_{r-1}\}$ and $\{v_{k-r+1}, \dots, v_{k}\}$ the \emph{start} and \emph{end} of the tight walk.
		This definition has the consequence that a subsequence $v_i v_{i+1} \dots v_j$ of $P$ need not be a tight walk. However, it has the useful property that given two tight walks $P, Q$ such that $P$ is from $S$ to $T$ and $Q$ is from $T$ to $R$, we can concatenate $P$ and $Q$ to get a walk $PQ$ from $S$ to $R$.

		We will need the following.
		\begin{definition}[Independent sets of leaves]
			For a forest $F$ we say that two sets of leaves $e,f\subseteq L(F)$ are \emph{independent} if there are vertices $v_e, v_f\in V(F)$ such that the vertices in $e$ are descendants of $v_e$, the vertices in $f$ are descendants of $v_f$, and $v_e$, $v_f$ are incomparable.
		\end{definition}

		\begin{figure}[h]
			\centering
			\includegraphics[scale = 1]{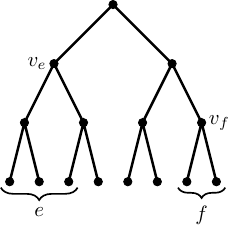}
			\vspace{-.3cm}
			\caption{Two independent sets of leaves $e,f$ and the corresponding ancestors $v_e,v_f$}
			\label{fig:independent}
		\end{figure}

		\begin{observation}\label{Observation_indepndent_sets_have_single_root}
			If $e,f$ are independent sets of leaves in an $S$-forest, then $|\pi_0(e)| = |\pi_0(f)|=1$.
		\end{observation}
		\begin{proof}
			Let all of $e$ descend from $v_e=(x,y)$. By definition of edges in $S$-forests, all descendants $v$ of $(x,y)$ have $\pi_0(v)=y$, showing that $|\pi_0(e)|=|\{y\}|=1$. The same works for $f$.
		\end{proof}

		A crucial definition in the paper is the following one --- it is the structure we find in our coloured graph without monochromatic tight paths.
		\begin{definition}[$k$-disconnected]
			Let $F$ be a $V(G)$-forest for a graph $G$, and let $\mathcal H$ be a coloured $r$-uniform hypergraph with $V(\mathcal H)=L(F)$. We say that a colouring of $\mathcal H$ is \emph{$k$-disconnected} on $(G,F)$ if for any $v\in V(G)$, any independent  $(r-1)$-sets $S,T\subseteq L(F(v))$ of the same type, and any monochromatic tight walk $P$ of length at most $3r$ from $S$ to $T$ in $\mathcal H$, we have $||V(P)||_{G}\geq k$.
		\end{definition}

		Height $1$ families are always disconnected.
		\begin{observation}\label{Observation_disconnected_star_family}
			Let $G$ be a graph and  $F$ be a $V(G)$-forest of height $1$ trees (i.e.\ stars). Then every hypergraph $\mathcal H$ of uniformity at least $3$ with $V(\mathcal H)=L(F)$ is $k$-disconnected on $(G,F)$ for every $k$.
		\end{observation}

		\begin{proof}
			Note that any $S,T\subseteq L(F(v))$ of size at least $2$ are not independent (since $|S|\geq 2$, the only common ancestor of the vertices in $S$ is the root $v$, and similarly for $T$. Thus the only choice of $v_S, v_T$ for the definition of ``independent'' could be $v_S=v, v_T=v$, which does not satisfy ``$v_S, v_T$ incomparable''). Therefore, the definition of ``$k$-disconnected'' holds vacuously since there are no pairs $S,T$ for which the definition needs to be checked.
		\end{proof}
		In contrast to this there do not exist $1$-disconnected families of large height. This will be proved via the following two lemmas.

		\begin{lemma} \label{lem:trees-tight-paths}
			Let $r, s \ll h \ll d$ and let $\ell \le r$.
			Let $T$ be a balanced $d$-ary ordered tree of height $h$, and let $\HH$ be the $r$-uniform complete graph on $L(T)$.
			Then for every $s$-colouring of $\HH$ there exist disjoint sets $X, Y, Z \subseteq L(T)$ of size $\ell$, such that $X,Z$ are independent and of the same type, and $X \cup Y$ and $Y \cup Z$ are monochromatic cliques of the same colour.
		\end{lemma}

		\begin{proof}
			Fix an $s$-colouring of all $r$-subsets of $L(T)$.
			By Lemma~\ref{Lemma_cleaning_forest} applied to $T$, there is a balanced binary subtree $T' \le T$ of height $h$, such that $r$-sets of leaves of the same type have the same colour. 
			Let $v_0$ be the first vertex in $L(T')$ (according to the ordering of $T'$, inherited from $T$), and let $v_0 \dots v_h$ be the path from $v_0$ to the root of $T'$. 
			Let $u_i$ be the last leaf in the tree rooted at $v_i$ for $i \in \{0, \dots, h\}$. (See \Cref{fig:ramsey-tree} for an illustration of the vertices $v_i$ and $u_i$.)
			By Ramsey's theorem there is a subset $A \subseteq \{u_0, \dots, u_h\}$ of size $2\ell+1$ whose $r$-subsets all have the same colour, say red. Let $w_1, \dots, w_{2\ell + 1}$ be the vertices in $A$, in order. 
			
			\begin{figure}[h]
				\centering
				\includegraphics[scale = .7]{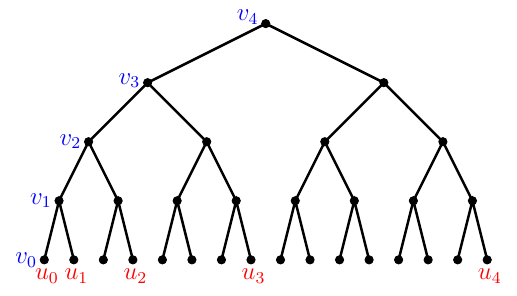}
				\hspace{.5cm}
				\includegraphics[scale = .7]{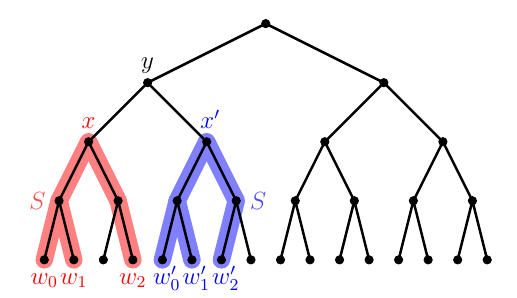}
				\vspace{-.3cm}
				\caption{An illustration of the sets $\{v_0, \dots, v_h\}$, $\{u_0, \dots, u_h\}$, $\{w_1, \dots, w_{\ell}\}$ and $\{w_1', \dots, w_h'\}$; here the leaves are ordered from left to right.}
				\label{fig:ramsey-tree}
			\end{figure}

			Let $S$ be the minimal ordered subtree of $T'$ whose leaves are $\{w_1, \dots, w_{\ell}\}$. Denote its root by $x$ and denote the father of $x$ by $y$. Let $x'$ be the child of $y$ which is not $x$, let $S'$ be a copy of $S$ rooted at $x'$, and denote the leaves of $S'$ by $\{w_1', \dots, w_{\ell}'\}$, noting that $\{w_1, \dots, w_{\ell}\}$ and $\{w_1', \dots, w_{\ell}'\}$ are independent (See \Cref{fig:ramsey-tree} for an illustration of the vertices $w_i$ and $w_i'$.) 

			Notice that $\{w_1, \dots, w_{\ell}, w_{\ell+2}, \dots, w_{2\ell + 1}\}$ and $\{w_1', \dots, w_{\ell}', w_{\ell+2}, \dots, w_{2\ell + 1}\}$ have the same type (we omitted $w_{\ell+1}$ as it may be a descendant of $y$, whereas $w_{\ell+1}, \dots, w_{2\ell+1}$ are not), so by choice of $w_1, \dots, w_{2\ell + 1}$ and $T'$ both sets are red cliques. Take $X = \{w_1, \dots, w_{\ell}\}$, $Y = \{w_{\ell+2}, \dots, w_{2\ell + 1}\}$ and $Z = \{w_1', \dots, w_{\ell}'\}$. These sets satisfy the requirements of the lemma.
		\end{proof}

		The following lemma is what gives a contradiction in our main proof.
		\begin{lemma}[Disconnection is bounded]\label{Lemma_disconnection_bounder}
			Let $r \le s \ll h \ll d$.
			Let $t \ge 0$ be an integer, let $G$ be a graph, and let $F$ be a $V(G)$-forest consisting of balanced $d$-ary trees of height at least $h$.
			Then there is no $s$-colouring of $\mathcal G^t\otimes_r F$ which is
			$1$-disconnected on $(G,F)$.
		\end{lemma} 
		\begin{proof}
			Suppose, for contradiction, that we have a $1$-disconnected, $s$-colouring of $\G^t\otimes F$.
			Consider some $v\in V(G)$ and the corresponding tree $F(v)\in F$. By Lemma~\ref{lem:trees-tight-paths}, there are disjoint $(r-1)$-sets $X=\{x_1, \dots, x_{r-1}\}$, $Y=\{y_1, \dots, y_{r-1}\}$, $Z=\{z_1, \dots, z_{r-1}\}\subseteq L(F(v))$ such that $X,Z$ are independent and of the same type, and $X\cup Y$, $Y\cup Z$ induce monochromatic complete hypergraphs. Now $P=x_1 \dots x_{r-1} y_1 \dots y_{r-1} z_1 \dots z_{r-1}$ is a monochromatic tight path in $\G^t \otimes F$ from $X$ to $Z$. Notice that for all $x\in V(P)$ we have $\pi_0(x)=v$. This gives $||e||_G = 0$ for every set of $r$ consecutive vertices in $P$, showing that $P$ is indeed a tight path in $\G^t \otimes F$. It also gives $||V(P)||_{G}=0$, showing that $\G^t\otimes F$ is not $1$-disconnected on $(F, G)$.
		\end{proof}
		 
		The following allows us to keep track of distances when we modify one forest into another. 
		 \begin{definition}[Short trees]
			 For a graph $G$ and a $V(G)$-forest $F$, we say that level $i$ is $k$-short in $G$ if for all its edges $xy$ between level $i-1$ and level $i$, we have $d_G(\pi(x),\pi(y))\leq k$.
		\end{definition}
		Note that this is equivalent to ``for all $v\in L(F)$, we have $d_G(\pi_{i-1}^F(v), \pi_{i}^F(v))\leq k$''.
		Short forests have the property that for two vertices $x,y$ in them it is easy to estimate $d_G(x,y)$.

		\begin{lemma}\label{Lemma_distances_in_short_trees}
			Let $G$ be a graph, and $F$ a $V(G)$-forest whose $1,\dots, \max\{i,j\}$ levels are $k$-short. For any set $W\subseteq L(F)$ and vertices $u\in \pi_i^F(W)$ and $v\in \pi_j^F(W)$, we have  $d_G(u,v)\leq ||W||_{G}+ik+jk$.
		\end{lemma}
		\begin{proof}
			Let $P=(u_0,x)(u_1,x), \dots, (u_{h},x)$ be a root-leaf path in $F$ such that $u_i = u$ and $(u_h, x) \in W$ (this exists by assumption on $u$). 
			Since levels $1, \dots i$ are $k$-short, we have that $d_G(u_{t-1}, u_{t})\leq k$ for all $t\leq i$, which gives $d_G(x,u)=d_G(u_0, u_i)\leq ik$.
			Similarly, there is a root-leaf path $Q = (v_0, y)(v_1, y), \dots, (v_h, y)$ such that $v_j = v$ and $(v_h, y) \in Q$, so $d_G(y,v) \le jk$.

			Also, since $(u_h,x), (v_h, y)\in W$, we have $d_G(x,y)=d_G\big(\pi_0(u_h,x), \pi_0(v_h, y)\big)\leq ||W||_G$. Using the triangle inequality, we get $d_G(u,v)\leq d_G(u,x)+d_G(x,y)+ d_G(y,v) \le ||W||_G+ik+jk$.
		\end{proof}

		The hypergraph we really care about is $\mathcal G^p$. The following is how we relate it to the hypergraphs $\mathcal G^t\otimes_r F$ for which all our machinery works.
		\begin{lemma}\label{Lemma_short_separated_homomorphism}
			Let $p\gg t,h,k,r$. 
			Let $G$ be a graph and $F$ a $V(G)$-forest of height at most $h$ that is $t$-separated on $G$ and whose levels are all $k$-short on $G$. Then $\pi: \mathcal G^t\otimes_r F \to \mathcal G^p$ maps edges to edges.
		\end{lemma}
		\begin{proof}
			Let $\{v_1, \dots, v_r\}\in E(\mathcal G^t\otimes F)$. By Observation~\ref{Observation_powers_distance_relation}, we have $||\{v_1, \dots, v_r\}||_{G}\leq t$, or, equivalently, $d_G(\pi_0(v_i), \pi_0(v_j)) \le t$ for every $i,j \in [r]$. Lemma~\ref{Lemma_distances_in_short_trees} implies $d_G(\pi(v_i), \pi(v_j))\leq t+2hk\leq p$ for all $i,j$. 
			Also, by $F$ being $t$-separated on $G$, we have $\pi(v_i)\neq \pi(v_j)$ for all distinct $i,j$ (indeed, write $u = \pi_0(v_i)$ and $v = \pi_0(v_j)$. If $u = v$ then $\pi(v_i) \neq \pi(v_j)$ because otherwise $v_i = v_j$. If $u \neq v$, since $d_G(u,v) \le t$ and $F$ is $t$-separated on $G$, then $\pi(F(u)) \cap \pi(F(v)) = \emptyset$). We have shown that $\{\pi(v_1), \dots, \pi(v_r)\}$ is a set of $r$ vertices in $G$ within distance $p$ of each other in $G$, i.e.\ an edge of $\mathcal G^p$.
		\end{proof}

	\subsection{Augmentations}\label{Section_augmentation}
		In this section we take a balanced forest $F$ of height $h$, and build a balanced forest of height $h+1$.  

		The following definition takes an $S$-forest and yields another $S$-forest with the roots of the former removed.
		\begin{definition}
			For an $S$-forest $F$ the \emph{root-deletion function} is the function $\rd^F:V(F)\setminus\{(x,x):x\in S\}\to S\times S$ by $\rd^F:(x,y)\to \big(x, \pi^{F}_1((x,y))\big)$. 
		\end{definition}


		\begin{observation}\label{Observation_rd_doesnt_change_on_subforests}
			If $F\leq F'$ are two forests, then restricting $\rd^{F'}$ to $F$ gives the function $\rd^F$, i.e.\ $\rd^{F'}|_F=\rd^F$.
		\end{observation}

		The following is how we build trees of larger heights. Informally, what the next definition does is take disjoint trees $T(u_1), \dots, T(u_k)$, and adds one extra vertex $v$ that is connected to all their roots, to get a new tree rooted at $v$.
		\begin{definition}[Augmentation of a single tree]\label{Definition_augment_one_tree}
			Let $u_1, \dots, u_k, v$ be distinct elements in a set $S$, and let $T_1, \dots, T_k$ be trees, such that $\root(T_i) = \{u_i\}$ for $i \in [k]$, and the sets $\pi\big(V(T_1)\big), \dots, \pi\big(V(T_k)\big), \{v\}$ are pairwise disjoint. Define an $S$-tree, denoted by $(v; T_1; \dots; T_k)$, as follows:
			\begin{itemize}
				\item 
					First, for each $i$, define $T^*_i$ by 
					\begin{align*}
						V(T^*_i) & =\{(x,v):(x,u_i)\in V(T_i)\} \\
						E(T^*_i) & =\{(x,v)(y,v): (x,u_i)(y,u_i)\in E(T_i)\},
					\end{align*}
					and order each $V(T^*_i)$ by $(x,v)<(y,v)\iff (x,u_i)<(y, u_i)$. 
				\item 
					Then, order  $\bigcup_{i \in [k]}V(T^*_i)$ by $V(T^*_1)< \dots < V(T^*_k)$ thereby making the union $\bigcup_{i \in [k]}T^*_i$ an ordered forest.
				\item  
					Finally, define $(v; T_1; \dots; T_k)$ to have vertex set $V$ and edge set $E$, defined as follows.
					\begin{align*}
						V &= \{(v,v)\}\cup \bigcup_{i \in [k]}V(T^*_i), \\
						E &= \{(v,v)(u_i,v): i \in [k]\} \cup \bigcup_{i \in [k]}E(T^*_i),
					\end{align*}
					and extend the ordering of $\bigcup_{i\in[k]}V(T^*_i)$ to an ordering of $V$ by letting $(v,v)$ precede all other vertices.
			\end{itemize}
		\end{definition}

		 It is perhaps not obvious that the above produces an ordered tree --- this is proved in the following lemma. Recall that $F^-$ denotes the subforest of $F$ with all roots deleted. 

		\begin{lemma}\label{Lemma_augment_one_tree_properties}
			Let $F$ be an $S$-forest, and let $u_1, \dots, u_k\in \pi_0(V(F))$ be such that the sets $\pi\big(V(F(u_1))\big), \dots, \pi\big(V(F(u_k))\big), \{v\}$ are pairwise disjoint. Set $T= (v; F(u_1); \dots; F(u_k))$.
			Then $T$ is an $S$-tree and $\rd^T|_{T^-}:T^-\to F$ is a monomorphism.
		\end{lemma}
		\begin{proof}
			With $T_i := F(u_i)$ and $T^*_i$ as in Definition~\ref{Definition_augment_one_tree}, note that $\phi:\bigcup_{i\in [k]} F(u_i)\to \bigcup_{i \in [k]}T^*_i$ defined by $\phi((x, u_i)) = (x,v)$ is an isomorphism.
			This shows that $\bigcup_{i \in [k]}T^*_i$ is an ordered forest with roots $(u_1, v), \dots, (u_k, v)$. The vertex $(v,v)$ is joined to exactly these vertices making $T$ a rooted tree. To see that the vertices of $T$ are ordered correctly we need to check that \ref{itm:ordered-a} and \ref{itm:ordered-b} hold when one of the vertices involved is $(v,v)$ (since we have already established that $T\setminus \{v\}$ is an ordered forest). Part \ref{itm:ordered-a} holds for $(v,v)$ since $v$ is an ancestor of everything and also ordered to be before everything.
			Part \ref{itm:ordered-b} does not need to be checked for $(v,v)$ since it concerns pairs of incomparable vertices --- but $(v,v)$ is comparable with everything.

			For the second part of this lemma, note that for $(x,v)\in V(T^*_i)$, we have that $\pi_1^{T}((x,v))=u_i$ (the vertex $(u_i,v)$ is on level 1 of $T$ because it is a child of the root $(v,v)$, also $(u_i,v)$ is an ancestor of $(x,v)$ since $(x,v)\in T^*_i$). This shows that for $(x,v)\in V(T^*_i)$ we have $\rd^T((x,v))=(x, u_i)=\phi^{-1}(x,v)$ --- and hence the function from $\bigcup_{i\in[k]}T^*_i$ to $\bigcup_{i\in[k]}F(u_i)$ given by $\rd^T$ is an isomorphism. Since $T^-= \bigcup_{i \in [k]}T^*_i$, we get that $\rd^T|_{T^-}:T^-\to F$ is a monomorphism.
		\end{proof}

		The following lemma works out how distances change with the above operation.
		\begin{lemma}\label{Lemma_Augmentation_shortness}
			For a graph $G$, let $F$ be a balanced $V(G)$-forest of height $h$, and let $u_1, \dots, u_k\in \pi_0(V(F))$ be chosen so that $\pi\big(V(F(u_1))\big), \dots, \pi\big(V(F(u_k))\big), \{v\}$ are pairwise disjoint. Set $T=(v; F(u_1); \dots; F(u_k))$. 
			For $i \in [h]$, if level $i$ of each $F(u_1), \dots, F(u_k)$ is $k$-short, then level $i+1$ of $T$ is $k$-short.
			Moreover, if $d(v, u_1), \dots, d(v, u_k)\leq k$, then level 1 of $T$ is $k$-short. 
		\end{lemma}

		\begin{proof}
			For the first part, note that from Definition~\ref{Definition_augment_one_tree}, with $T_i = F(u_i)$ and $T_i^*$ as in \Cref{Definition_augment_one_tree}, we get that vertices of $T$ on level $i$ are the vertices of $\bigcup_{i\in[k]}T^*_i$ on level $i-1$. Let $\phi : \bigcup_{i \in [k]}F(u_i) \to \bigcup_{i \in [k]}T^*_i$ be the function from the proof of the previous lemma, which was shown to be an isomorphism. For a vertex $x$ in $\bigcup_{i\in[k]}T^*_i$, we have $\phi^{-1}(x)$ is a vertex of $\bigcup_{i\in[k]}F(u_i)$ with $\level(x)= \level(\phi^{-1}(x))$ (by Lemma~\ref{Lemma_level_preserved_by_monomorphism}). Thus 
			\begin{align*}
				& \max\{d_G\big(\pi(x),\pi(y)\big): \text{$xy\in E(T^{*}_i)$, with $x$ on level $i-1$ and $y$ on level $i$}\} \\
				=\, & \max\{d_G\big(\pi(x),\pi(y)\big): \text{$xy\in E(F(u_i))$, with $x$ on level $i-1$ and $y$ on level $i$}\} \leq k.
			\end{align*}
			This shows that level $i$ in $\bigcup_{i \in [k]}T_i^*$ is $k$-short, implying that level $i+1$ in $T$ is $k$-short, for every $i \in [h]$.

			For the second part note that from Definition~\ref{Definition_augment_one_tree}, the only edges of $T$ from level 0 to level 1 are the edges $\{(v,v)(u_i,v): i \in [k]\}$, and $d_G(v,u_i)\leq k$, showing that level 1 is $k$-short.
		\end{proof}

		Roughly speaking, we think of one forest $F'$ is an augmentation of another forest $F$ if $F'\leq F''$ for some forest $F''$, whose trees are constructed from trees of $F$ via Definition~\ref{Definition_augment_one_tree}. This informal definition is cumbersome to work with, so we instead work with the following definition via the $\rd$ function. 

		\begin{definition}[Augmentation of tree families]
			Let $F, F'$ be $S$-forests. 
			Say $F'$ \emph{augments} $F$ if every $v\in S$ satisfies $\rd^{F'}(F'(v)^-)\subseteq V(F)$  and moreover $\rd^{F'}|_{F'(v)^-}:F'(v)^- \to F$ is a monomorphism (see \Cref{fig:augment}).
		\end{definition}
		\begin{figure}[h]
			\centering
			\includegraphics[scale = 1]{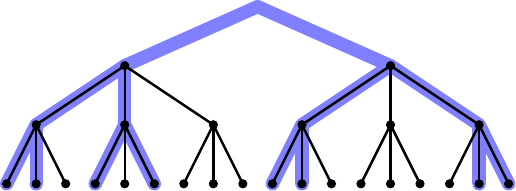}
			\caption{An illustration of a forest $F'$ (in blue) which is an augmentation of a forest $F$ (in black)}
			\label{fig:augment}
		\end{figure}

		The following lemma gives us the properties of augmentation we use.
		\begin{lemma}[Properties of augmentation]\label{Lemma_augmentation_properties}
			Let $G$ be a graph and $F, F'$ be $V(G)$-forests with $F'$ augmenting $F$. The following are true.
			\begin{enumerate}[label = \rm(\roman*)]
				\item \label{itm:augment-a}
					$\pi^{F}\circ \rd^{F'} = \pi^{F'}$ and $\pi_1^{F'}|_{F'^-} = \pi_0^F\circ \rd^{F'}$.
				\item \label{itm:augment-d}
					If level 1 of $F'$ is $k$-short on $G$, then for any $S\subseteq L(F)$, we have $||S||_{G}\geq ||\rd^{F'}(S)||_{G}-2k.$  
				\item \label{itm:augment-e}
					If $t \ge s+2k$, $F'$ is $s$-separated on $G$, and its level 1 is $k$-short on $G$, then the function
					$\rd^{F'}:\mathcal G^s\otimes_r F'\to \mathcal G^t\otimes_r F$ maps edges to edges.
			\end{enumerate}
		\end{lemma}

		\begin{proof}
			\hfill
			\begin{enumerate}[label = \rm(\roman*)]
				\item 
					The equality $\pi^{F}\circ \rd^{F'}=\pi^{F'}$ is immediate from the fact that $\pi$ is always just a projection on the first coordinate and since $\rd$ does not change the first coordinate. Meanwhile, $\pi_1^{F'}|_{F'^-}=\pi_0^F\circ \rd^{F'}$ is immediate from the definition of $\rd^{F'}$.
				\item   
					By definition, $\tau=\type_{F'}(X)$ is the isomorphism class of $A^{F'}(X)$, and $\type_{F}(\rd^{F'}(X))$  is the isomorphism class of $A^F(\rd^{F'}(X))$. 
					By \Cref{Lemma_A_preserved_by_monomorphisms}, since $\rd^{F'}|_{F'(v)^-}:F'(v)^-\to F$ is a monomorphism, we have $A^{F}(\rd^{F'}(X)) = \rd^{F'}(A^{F'(v)^-}(X)) = \rd^{F'}(A^{F'^-}(X))$, so $\type_{F}(\rd^{F'}(X))$ is the isomorphism class of $A^{F'^-}(X)$.
					By Lemma~\ref{Lemma_A_preserved_by_minus}, we also have $A^{F'^-}(X) = A^{F'}(X)^-$, and by definition of $\tau^-$, the isomorphism class of $A^{F'}(X)^-$ is $\tau^-$. Altogether, $\type_F(\rd^{F'}(X)) = \tau^-$, as required.

				\item 
					Recall that edges of $\mathcal G^s\otimes F'$ are the $r$-sets $e\subseteq L(F')$ which have $||e||_{G}\leq s$. Part \ref{itm:augment-d} shows that $||\rd^{F'}(e)||_{G}\leq s+2k\leq t$ --- therefore $\rd^{F'}(e)$ is contained in an edge of $\mathcal G^t\otimes F$. Moreover, since $F'$ is $s$-separated, every edge $e\in \mathcal G^s\otimes F'$ satisfies $|\pi(e)|=|e|=r$. Indeed, consider $u,v \in e$ distinct. If $\pi_0(u) = \pi_0(v)$ then $\pi(u) \neq \pi(v)$ by distinctness of $u,v$; if instead $\pi_0(u) \neq \pi_0(v)$ then $\pi\big(F(\pi_0(u))\big) \cap \pi\big(F(\pi_0(v))\big) = \emptyset$, showing $\pi(u) \neq \pi(v)$. Using \ref{itm:augment-a}, we have $|\rd^{F'}(e)|\geq |\pi(\rd^{F'}(e))|=|\pi(e)|=r$, which  tells us that $|\rd^{F'}(e)|=r$, and so $\rd^{F'}(e)$ is an edge in $\G^s \otimes F'$. \qedhere
			\end{enumerate}
		\end{proof}
	  
	\subsection{Versatile sets} \label{Section_versatile}
		The goal in this section is to set up a framework under which we can concatenate tight paths/walks. This is based around the following concept. 
		\begin{definition}[Versatile sets] \label{def:versatile}
			For a balanced, height $h$ ordered tree $T$, a subset $e\subseteq L(T)$ is \emph{$t$-versatile} in $T$ if for every balanced, height $h$ ordered tree $T'$ with at most $t$ leaves, and monomorphism $\phi:A^T(e)\to T'$, there is a monomorphism $\psi:T'\to T$ with $\psi \circ\phi=\mathrm{id}$.
		\end{definition}

		We will prove that versatile sets exist inside all balanced $d$-ary trees with large enough $d$. This will involve an inductive argument where trees are built one vertex at a time via adding leaves. We first set up the theory of extending ordered trees like this. We remark that, unlike the rest of the paper, in this section we do not require trees to be balanced. 

		In a tree $S$, we call a path $P=v_1 \dots v_k$ \emph{extendible} if $v_1 < \dots < v_k$ (i.e.\ $v_{i-1}$ is the parent of $v_i$ for $i \in \{2,\dots,k\}$) and there is no vertex $x$ satisfying $x>v_k$ which is comparable with $v_1$ (see \Cref{fig:extendible}). 
		It is useful to note that isomorphisms map extendible paths to extendible paths.

		\begin{figure}[h]
			\centering
			\includegraphics[scale = 1]{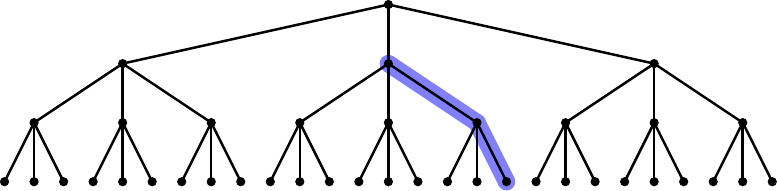}
			\caption{An extendible path in an ordered trees (with ordering of vertices in each level from left to right).}
			\label{fig:extendible}
		\end{figure}

		\begin{observation} \label{obs:extendible}
			A path $P = v_1 \dots v_k$ in an ordered tree $S$ is extendible if and only if $v_i$ is the latest child of $v_{i-1}$ in the ordering of $V(S)$, for $i \in \{2,\dots,k\}$, and $v_k$ is a leaf.
		\end{observation}

		\begin{proof}
			For the ``if'' part, by definition we have $v_1 < \dots < v_k$, showing that $v_i$ is a child of $v_{i-1}$, for $i \in \{2,\dots,k\}$. Let $u$ be the last descendent of $v_1$ according to the ordering of $V(S)$. Then $u$ is a leaf, and $v_k = u$, implying that $v_i$ is the latest child of $v_{i-1}$ for every $i$, by \ref{itm:ordered-b}.

			For the ``only if'' part, suppose that $P$ has the property given in the observation, and suppose that $P$ is not extendible. This means that there is a vertex $x$ which satisfies $x > v_k$ and $x$ is comparable with $v_1$. If $x$ is an ancestor of $v_1$ then $x < v_1 < v_k$, a contradiction. If $x$ is a descendant of $v_1$ then $x \le v_k$, because $v_k$ is the last descendant of $v_1$, again a contradiction.
		\end{proof}

		Let $S\le T$ be ordered trees and $P$ an extendible path in $S$. We call $T$ a $P$-extension of $S$ if $T$ has one extra vertex $y$, one extra edge $v_1y$, and $y$ is the successor of $v_k$ in $T$ (i.e.\ $y$ is immediately after $v_k$ in the ordering on $V(T)$).
		\begin{lemma}\label{Lemma_tree_add_leaf}
			Let $S\leq T$ be ordered trees sharing a root so that $T$ has one extra vertex. Then $T$ is a $P$-extension of $S$ for some extendible path $P$ in $S$.
		\end{lemma}
		\begin{proof}
			Let $y$ be the vertex of $V(T)\setminus V(S)$, noting that it must be a leaf. Let $v_1y$ be the edge of $E(T)\setminus E(S)$, and let $v_0$ be the predecessor of $y$ in the ordering of $V(T)$. We claim that $v_0, v_1$ are comparable. Indeed, if not, then by part \ref{itm:ordered-a} of the definition of ordered trees, we have $y > v_1$, and by choice of $v_0$ we have $v_1 < v_0 < y$. Thus, by part \ref{itm:ordered-b} of the same definition, with $u = v_1, u' = y, v = v' = v_0$, we get $y < v_0$, a contradiction. So $v_0, v_1$ are comparable, which in fact implies that $v_0$ is a descendant of $v_1$ (because otherwise $v_0 < v_1 < y$, contrary to the choice of $v_0$ as $y$'s predecessor). Thus, there is a path $P = v_1 \dots v_k = v_0$ in $S$, with $v_1 < \dots < v_k$.

			We claim that $P$ is extendible. To see this, consider a vertex $x$ satisfying $x>v_k$, and suppose for contradiction that $x, v_1$ are comparable. If $x$ is an ancestor of $v_1$, then $x < v_1 < v_k$, a contradiction. If $x$ is a descendant of $v_1$, then $v_1, y$ are incomparable (since $y$ is a leaf, it is only comparable with $v_1$ and its ancestors). 
			Then \ref{itm:ordered-b} with $u=v_1, u'=x, v=v'=y$ gives $x<y$, a contradiction to $v_k$ being the predecessor of $y$. 
		\end{proof}
		Repeatedly applying the above lemma gives the following.
		\begin{corollary}\label{Corollary_trees_constructed_by_extensions}
			Let $S\leq T$ be ordered trees sharing a root. Then $T$ can be constructed from $S$ via a sequence of $P$-extensions.
		\end{corollary}
		 
		Extensions can be used to build monomorphisms as follows.
		\begin{observation}\label{Observation_extend_isomorphism}
			Let $S,T,S',T'$ be ordered trees, and $P_S,P_T$ paths in $S, T$, such that $S'$ is a $P_S$-extension of $S$, and $T'$ is a $P_T$-extension of $T$.
			If  $f: S \to T$  an isomorphism with $f(P_S)=P_T$, then $f$ can be extended to an isomorphism $f: S'\to T'$.
		\end{observation}
		\begin{proof}
			Let $s''$ be the vertex of $V(S')\setminus V(S)$, and $t''$ the vertex of $V(T')\setminus V(T)$. Extend $f$ by defining $f(s'')=t''$. This is clearly a bijection, maps edges to edges (for this notice that all vertices in $P_S$ precede $s''$ in $S'$, and so if $P_S$ has ends $a,b$ with $a < b$, then $b$ is the predecessor of $s''$, showing that $s''$ is adjacent to $a$. A similar argument holds for $P_T$, and then the claim follows by the fact that $f$ is an isomorphism which maps $P_S$ to $P_T$), and maps the predecessor of $s''$ to the predecessor of $t''$, which implies that $f$ is order preserving, i.e.\ it is an isomorphism.
		\end{proof}

		Our main building block for versatile sets is the following lemma.
		\begin{lemma}\label{Lemma_find_tree_with_property_ast}
			Let $d\gg h,k, s$.
			Let $T_0$ be a balanced $d$-ary tree of height $h$. There is a sequence of trees $T_0\ge \dots \ge T_k$ such that $T_k$ is a balanced $s$-ary tree of height $h$ and we have the following property:
			\begin{enumerate}[label = \rm($\ast$)]
				\item \label{itm:star}
					For any $S \le T_i$ and extendible path $P$ in $S$, there is a $P$-extension of $S$ contained in $T_{i-1}$.
			\end{enumerate}
		\end{lemma} 
		\begin{proof}
			Without loss of generality, we can assume that $d=2^k(s+1)-1$ (otherwise pass to a subtree which is $(2^k(s+1)-1)$-ary and continue the proof there).
			For each $i \in [k]$, set $d_i=2^{k-i}(s+1)-1$, noting that $d_k=s$ and $d_0=d$.
			Label $V(T)$ by $([d] \cup \{0\})^h$ as explained in Notation~\ref{not:ordered-tree}. 
			Define 
			\begin{equation*}
				X_i=2^i \cdot [d_i] = \{2^ij : j \in [d_i]\},
			\end{equation*}
			noting that $[d]=X_0\supseteq \dots \supseteq X_k$. 
			Let $T_i$ be the subtree of $T$ consisting of vertices in $(X_i \cup \{0\})^h$, noting that $T_i$ is a $d_i$-ary tree, and we have $T=T_0\ge \dots \ge T_k$.  
			For property \ref{itm:star}, consider some $S\le T_{i}$ and extendible path $P$ in $S$. Then $P$ can be written as $v_0 \dots v_{\ell}$, with $v_0 < \dots < v_{\ell}$, where  
			\begin{align*}
				v_0& \, =(2^i a_1, \dots, 2^i a_t, 0, \dots, 0) \\
				v_\ell& \, =(2^i a_1, \dots, 2^i a_t, 2^i a_{t+1}, \dots, 2^i a_{t+\ell},0, \dots, 0),
			\end{align*}
			for some $t,\ell \in [h]$ and $a_1, \dots, a_{t+\ell} \in [d_i]$.
			Define 
			\begin{equation*}
				y=(2^ia_1, \dots, 2^ia_t,2^ia_{t+1}+2^{i-1},0,\dots,0).
			\end{equation*}
			Notice that $y \in V(T_{i-1}) \setminus V(T_i)$ (as $2^i[d_i]+2^{i-1} \subseteq 2^{i-1}[d_{i-1}]$, using $d_{i-1} \ge 2d_i + 1$) and that $y$ is a child of $v_0$ in $T_{i-1}$, and add $y$ to $S$ to get a $P$-extension. To see that this is indeed a $P$-extension, we need to show that $y$ is the successor of $v_\ell$ in $S$. 
			First, observe that $v_\ell < y$ (as the first $t$ coordinates of $v_\ell$ and $y$ are the same, and the $(t+1)$-st coordinate of $v_\ell$ is smaller than that of $y$). It thus suffices to show that there is no vertex $x \in V(S)$ satisfying $v_\ell < x < y$.
			Suppose there is such $x=(x_1, \dots, x_h)$. The first coordinate where $x$ disagrees with $v_\ell$ and $y$ must be $t+1$ --- otherwise we would have $x<v_\ell$ or $x>y$. On the $(t+1)$-st coordinate we would need $2^i a_t < x_{t+1} < 2^i a_t + 2^{i-1}$ --- but then $2^{-(i-1)}x_{t+1}$, which is an integer by definition of $X_{i-1}$, satisfies $2a_t < 2^{-(i-1)}x_{t+1} < 2a_t + 1$, a contradiction.
		\end{proof}

		The following is one of the key results in this section.

		\begin{lemma}[Existence of versatile sets]\label{Lemma_versatile_set_existence}
			Let $d\gg h,s,t$.
			Let $T$ be a balanced $d$-ary tree of height $h$ and let $\tau\in \types(h,s)$. There is subset $e\subseteq L(T)$ of type $\tau$ that is $t$-versatile. 
		\end{lemma} 
		\begin{proof}
			Consider a sequence $T=T_0\ge \dots \ge T_{ht}$, as given by Lemma~\ref{Lemma_find_tree_with_property_ast}. 
			Let  $S$ be a type $\tau$ subtree of $T_{ht}$ (which exists because $T_{ht}$ is a balanced $s$-ary tree of height $h$), and set $e=L(S)$ (so $S=A^T(e)$ and $e$ is type $\tau$). Consider a height $h$ tree $T'$ with at most $t$ leaves and a monomorphism $\phi:S\to T'$ as in \Cref{def:versatile} of ``$t$-versatility''. 
			We need to find a monomorphism $\psi : T' \to T$ such that $\psi \circ \phi = \id_S$.

			Set $k:=|T'|-|S|$, noting $k\leq ht$.
			Use Lemma~\ref{Corollary_trees_constructed_by_extensions} to get a sequence of trees $\phi(S)=A_0'\le A_1'\le \dots \le A_{k}'=T'$ and paths $P_i'\subseteq A_i'$ such that each $A_{i+1}'$ is constructed from $A_i'$ via a  $P_i'$-extension. 

			We build trees $A_i\subseteq T_{ht-i}$ and isomorphisms $\psi_{i}:A_i'\to A_i$, for $i \in \{0, \dots, k\}$, as follows. Start with $A_0:=S$ and $\psi_0=(\phi|_{\phi(S)})^{-1}$.
			Now assume that we have defined $A_0, \dots, A_i$ and $\psi_0, \dots, \psi_i$ with the desired properties.
			Since $A'_{i+1}$ is built from $A'_{i}$ via a $P'_i$-extension, $\psi_i(P_i')$ is extendible in $A_i$. Thus, using \ref{itm:star} and $A_i \le T_{ht-i}$, there is a $\psi_i(P_i')$-extension $A_{i+1} \le T_{ht-i-1}$ of $A_i$.
			Use \Cref{Observation_extend_isomorphism} to extend $\psi_i$ to an isomorphism $\psi_{i+1}:A'_{i+1}\to A_{i+1}$.

			We claim that $\psi_k$ satisfies the requirements, namely it is a function from $T'$ to $T$ (this holds by construction) and $\psi \circ \phi = \id_S$. For the latter bit, by construction, we have $\psi_k|_{\phi(S)}=\psi_k|_{A_0'}=\psi_0=(\phi|_{\phi(S)})^{-1}$. Therefore $\psi_k(\phi(x))=\phi^{-1}(\phi(x))=x$ for all $x \in V(S)$, as required.
		\end{proof}

		The following lemma is how we use versatile sets. Given an arbitrary tight walk $\mathcal P$, we get another one that goes between two versatile sets --- without changing the types of edges in $\mathcal P$.
		\begin{lemma}\label{Lemma_path_between_versatile_sets}
			Let $\HH$ be an $r$-uniform hypergraph, let $F$ be a balanced $V(\HH)$-forest and $\mathcal P$ a tight walk in $\HH \otimes F$ with $V(\mathcal P)\subseteq L(F)$ and write $\P = v_1 \dots v_k$.
			Write $a = (v_1, \dots, v_{r-1})$ and $b = (v_{k-r+2}, \dots, v_k)$, and suppose that $u_a, u_b$ are distinct roots of $F$ such that $V(a) \subseteq L(F(u_a))$ and $V(b) \subseteq L(F(u_b))$.

			Suppose that $a'\subseteq L(F(u_a))$ and $b'\subseteq L(F(u_b))$ are $k$-versatile in $F(u_a)$ and $F(u_b)$, respectively, and satisfy $\type_F(a')=\type_F(a)$ and $\type_F(b')=\type_F(b)$. Then there is a tight walk $\mathcal P'$ in $\HH \otimes F$ from $a'$ to $b'$ of order $k$ with the property that the $i$-th edge of $\mathcal P'$ has the same root set and type as the $i$-th edge of $\mathcal P$.
		\end{lemma}
		
		\begin{proof}
			Let $A=A^{F}\big(V(\mathcal P)\cap V(F(u_a))\big)$ and $B=A^F\big(V(\mathcal P)\cap V(F(u_b))\big)$, and note that $A,B$ are trees with at most $k$ leaves. 
			Let $\phi_a:A^F(a') \to  A^F(a)$ and $\phi_b:A^F(b') \to  A^F(b)$ be isomorphisms (which exist because $a,a'$ and $b,b'$ have the same types). Noting that $A^F(a)\subseteq A$ and $A^F(b)\subseteq B$, we can use \Cref{def:versatile} of ``versatility'' to get monomorphisms $\psi_a: A\to F(u_a)$ and $\psi_b:B\to F(u_b)$ with $\psi_a\circ \phi_{a} = \id_{A^F(a')}$ and $\psi_b\circ \phi_b=\id_{A^F(b')}$. 

			Let $\psi:A^F(V(\mathcal P))\to V(F)$ be the function which is $\psi_a$ on $A$, $\psi_b$ on $B$, and the identity everywhere else (this is well defined as $A$ and $B$ are vertex-disjoint by assumption on $u_a,u_b$).  Note that $\psi$ is the identity on roots and a monomorphism (if $\psi$ fixes roots, and $\psi|_{F(r)}$ is a monomorphism for every $r$ then it is a monomorphism). Denote by $e_i$ the $i$-th edge of $\P$, and write $e_i' = \psi(e_i)$. 
			We claim that $e_i'$ is an edge in $\HH \otimes F$. Indeed, because $\psi$ is injective we have $|e_i'| = |\psi(e_i)| = |e_i| = r$. Moreover, $\pi_0(e_i') = \pi_0(\psi(e_i)) = \pi_0(e_i)$, because $\pi_0(\psi(x)) = \pi_0(x)$ for every vertex $x$ in the range of $\psi$ (as this holds for $\psi_a$, $\psi_b$, and the identity). By definition of $\HH \otimes F$, since $e_i$ is an edge we have $\pi_0(e_i)$ is a subset of an edge in $\HH$, and thus the same holds for $e_i'$, implying that $e_i'$ is an edge in $\HH \otimes F$. Define $\P' = \psi(\P)$. Then, by $e_i'$ being an edge in $\HH \otimes F$ for every $i$, we have that $\P'$ is a tight path in $\HH \otimes F$, whose $i$-th edge is $e_i'$.

			Lemma~\ref{Lemma_A_preserved_by_monomorphisms} gives $\psi(A^F(e_i))=A^F(\psi(e_i))=A^F(e_i')$, and hence $e_i, e_i'$ have the same type. We have already seen that the root set of $e_i'$ has the same root set as $e_i$. Also, $\psi(a)=\psi(\phi_a(a'))=\psi_a(\phi_a(a'))=a'$ and $\psi(b)=\psi(\phi_b(b'))=\psi_b(\phi_b(b'))=b'$, by construction, showing that $\mathcal P'$ satisfies the requirements of the lemma.
		\end{proof}


\section{Proof of the main theorem}\label{Section_proof}
	The following is the main technical result of the paper --- Theorem~\ref{thm:main-path0} will follow from it via a short argument.
	Here a \emph{tight walk} of order $k$ in an $r$-uniform hypergraph $\HH$ is a sequence $v_1 \dots v_k$ of vertices in $\HH$ that need not be distinct, such that $(v_i, \dots, v_{i+r-1})$ is an edge in $\HH$ for $i \in [k-r+1]$.
	The \emph{degree} of a vertex $v$ in a hypergraph is the number of edges containing $v$.

	\begin{theorem}\label{Theorem_walk}
		For all integers $r\geq 3$ and $s \ge 2$, there exists an integer $p$ such that: for all sufficiently large $n$, there exists an $r$-uniform hypergraph $\mathcal H$ on $n$ vertices, with maximum degree at most $p$, such that in every $s$-colouring of $\mathcal H$ there is a monochromatic tight walk of order at least $n/p$ in which every vertex is used at most $p$ times.
	\end{theorem} 

	\begin{proof}
		Use Lemma~\ref{Lemma_disconnection_bounder} to pick a number $h:=h(r,s)$ satisfying that lemma (i.e.\ if $F$ is a balanced $d$-ary $V(G)$-forest of height $h$ then no $s$-colouring of $\G^t \otimes_r F$ is $1$-disconnected, for every non-negative integer $t$).
		Pick numbers $p,p',\eps,a_1,b_1,c_1,d_1,\dots,a_h,b_h,c_h,d_h$, as follows.
		\begin{align*}
			p\gg p'\gg  d_1\gg b_1\gg c_1\gg a_1 & \gg d_2\gg b_2\gg c_2\gg a_2\gg \dots \\
			\dots & \gg d_h\gg b_h\gg c_h\gg a_h\gg \eps^{-1}\gg h, r,s.
		\end{align*}
		Use Lemma~\ref{Lemma_expander_existence} to obtain a connected graph $G$ on $n$ vertices which is an $\eps$-expander with maximum degree at most $\Delta:=\eps^{-2}$. Order $V(G)$ arbitrarily.
		Let $\HH$ be the $r$-uniform hypergraph $\mathcal G^{p'}$, noting that $\Delta(\HH)\leq \binom{(\Delta+1)^{p'}}{r}\leq p$. 
		Fix an $s$-colouring of $\HH$, and suppose for contradiction that there is no monochromatic tight walk of length at least $n/p$ in $\HH$ in which every vertex is used at most $p$ times.

		Throughout the proof, all graphs of the form $\mathcal G^t\otimes F$ we consider satisfy $t\leq d_1$, and $F$ being balanced, $t$-separated, of height at most $h$, and having all levels $d_1$-short. 
		For such hypergraphs, Lemma~\ref{Lemma_short_separated_homomorphism} tells us that $\pi: \mathcal G^t\otimes F \to \mathcal G^{p'}$ is a hypergraph homomorphism (i.e.\ maps edges to edges). This gives a natural $s$-colouring on $\mathcal G^t\otimes F$ where we colour $e\in \mathcal G^t\otimes F$ by the colour of $\pi(e)$ in $\HH = \mathcal G^{p'}$. Note that, this way, the following holds. 
		\begin{enumerate}[label = \ding{71}]
			\item \label{itm:image-walk}
				If $W$ is a monochromatic tight walk in  $\mathcal G^t\otimes F$, with $t \le d_1$, then $\pi(W)$ is a monochromatic tight walk in $\HH$.
		\end{enumerate}

		We will construct subsets $V(G)=U_1\supseteq \dots\supseteq U_h$, and $V(G)$-forests $F_1, \dots, F_h,$ with the following properties holding for every $i \in [h]$.
		\begin{enumerate}[label = \rm(F\arabic*)]
			\item \label{itm:forest-size-U}
				$|U_i|\geq (1-200\eps)^{i}n$ (which in particular implies $|U_i|\geq (1-200h\eps)n\geq n/2$).
			\item \label{itm:forest-height}
				Trees in $F_i$ are balanced $c_i$-ary of height $i$ with $\pi_0(F_i)=U_i$. 
			\item \label{itm:forest-separated}
				All levels in $F_i$  are $d_1$-short on $G$, and $F_{i}$ is $b_i$-separated on $G$. 
			\item \label{itm:forest-clean}
				The hypergraph $\mathcal G^{c_i}\otimes F_{i}$  is cleanly coloured.
			\item \label{itm:forest-disconnected}
				The colouring on $\mathcal G^{c_i}\otimes F_{i}$  is $c_i$-disconnected on $(G, F_i)$.
			\item \label{itm:forest-augmentation}
				$F_{i}$ is an augmentation of $F_{i-1}$, for $i\geq 2$.
		\end{enumerate}
		We will show that under the assumption that $\HH$ has no monochromatic tight walk of length at least $n/p$ where every vertex is used at most $p$ times, then sequences as above can indeed be constructed.
		In particular, we have a forest $F_h$ which is a balanced $d_h$-ary $V(G)$-forest of height $h$, such that $\G^{d_h} \otimes F_h$ is $1$-disconnected on $(G, F_h)$. This is a contradiction to the choice of $h$, proving the theorem. 

		We begin by constructing $U_1, F_1$ satisfying \ref{itm:forest-size-U} -- \ref{itm:forest-disconnected}.
		Set $U_1=V(G)$ (ensuring \ref{itm:forest-size-U}). Noting that  ``$G$  connected $\implies$ $\delta(G^{d_1})\geq d_1$'', for each $v\in V(G)$ pick distinct neighbours $u_1^v, \dots, u_{d_1}^v$ of $v$ in $G^{d_1}$. Define $F_1''(v)$ by $V(F_1''(v))=\left\{(v,v), (v, u_1^v), \dots, (v,u_{d_1}^v)\right\}$, with $(v,v)$ connected to all other vertices, making it an ordered $V(G)$-star (with the ordering of the leaves inherited from $V(G)$, and $(v,v)$ preceding all leaves). Taking the union of these (and ordering them according to the ordering of the roots, inherited from the ordering of $V(G)$), we get a $V(G)$-forest $F_1''$ of (balanced) $d_1$-ary trees of height $1$.
		Note that each $F_1''(v)$ has all levels $d_1$-short in $G$ since it is a subgraph of $G^{d_1}$, implying that $F_1''$ has all levels $d_1$-short.
		Apply Lemma~\ref{Lemma_separate_trees} (with $(F,G,d,b,h,\Delta)_{\ref{Lemma_separate_trees}} = (F_1'',G,d_1,b_1,1,\Delta)$) to get a subforest $F_1'\leq F_1''$ of balanced $b_1$-ary trees of height $1$ which is $b_1$-separated on $G$.
		Now apply Lemma~\ref{Lemma_cleaning_forest} (with $(F, \HH, d,b,h,r,s,\Delta)_{\ref{Lemma_cleaning_forest}} = (F_1',\G^{c_1},b_1,c_1,1,r,s,\binom{(\Delta+1)^{c_1}}{r})$) to get a $c_1$-ary subforest $F_1$ satisfying \ref{itm:forest-size-U} to \ref{itm:forest-clean}. Property \ref{itm:forest-disconnected} holds by Observation~\ref{Observation_disconnected_star_family}.  

		Now suppose that for some $i \in [h-1]$, we have constructed $U_i, F_i$ satisfying \ref{itm:forest-size-U} to \ref{itm:forest-disconnected}. 
		We will show how to construct $U_{i+1}$ and $F_{i+1}$ satisfying \ref{itm:forest-size-U} to \ref{itm:forest-augmentation}, thereby proving the theorem.
		Consider an auxiliary colouring on the graph $G^{a_{i}}[U_i]$ with colour set $([s]\times \types(i,r-1))\cup \{\grey\}$, such that each edge $uv$ is coloured as follows.
		\begin{itemize}
			\item 
				Colour $uv$ with colour $(c,\tau) \in [s] \times \types(i,r-1)$ if there are $(r-1)$-sets $U \subseteq L(F_i(u))$ and $V \subseteq L(F_i(v))$, both of type $\tau$, and a colour $c$ tight walk $P$ in $\G^{c_i} \otimes F_i$ from $U$ to $V$ of length at most $3r$, which satisfies $||P||_{G} \le c_i$.
			\item 
				If the previous item fails for all $(c,\tau) \in [s] \times \types(h,r-1)$, colour $uv$ $\grey$.
		\end{itemize} 
		(Notice that edges $uv$ could have more than one colour in $[s] \times \types(i,r-1)$.)
		By Lemma~\ref{Lemma_expander_ramsey_MAIN} (applied to $G$ with $(U,c,d,\eps,s)_{\ref{Lemma_expander_ramsey_MAIN}} = (U_i,a_i,d_{i+1}+1,\eps, s\cdot |\types(i,r-1)|)$), either we get a non-grey monochromatic path of length $n/a_{i}$, or we can cover at least $(1-200\eps)|U_{i}|$ vertices by vertex-disjoint grey $K_{d_{i+1}+1}$'s. 

		\begin{claim} \label{claim:non-grey-mono-path}
			There is no non-grey monochromatic path of length at most $n/a_i$.
		\end{claim}
		\begin{proof}
			Suppose to the contrary that $P = p_1 \dots p_m$ is a monochromatic non-grey path of length at least $n/a_i$ in the above colouring.
			Let $(c, \tau)\in [s]\times \types(i,r-1)$ be the colour of $P$.
			For each $j \in [m]$, use Lemma~\ref{Lemma_versatile_set_existence} (with $(T,\tau,d,h,s,t)_{\ref{Lemma_versatile_set_existence}} = (F_i(p_j),\tau,c_i,i, r-1,3r)$) to pick a type $\tau$ set $A_j \subseteq L(F_i(p_j))$ which is $3r$-versatile in $F_i(p_j)$. 
			For each $j \in [m-1]$, since $p_jp_{j+1}$ is a colour $(c, \tau)$ edge, there is a colour $c$ tight walk $P_j$ in $\mathcal G^{c_i}\otimes F_i$ from a type $\tau$ subset $U_i\subseteq L(F_i(p_j))$ to a type $\tau$ subset $V_j\subseteq L(F_i(p_{j+1}))$, whose length is at most $3r$ and which satisfies $||P_j||_{G}\leq c_{i}$. 
			By Lemma~\ref{Lemma_path_between_versatile_sets} (with $(F,\HH,\mathcal{P},u_a,u_b,a,b,a',b',k,r)_{\ref{Lemma_path_between_versatile_sets}} = (F_i,\G^{c_i},P_j,p_j,p_{j+1},U_j,V_j,A_j,A_{j+1},3r,r)$), there is a tight walk $Q_j$ from $A_j$ to $A_{j+1}$, of the same length as $P_j$, and whose edges have the same types and roots as corresponding edges of $P_j$. Since the colouring of $G^{c_i}\otimes F_i$ is clean (by \ref{itm:forest-clean}), we get that $Q_j$ is also colour $c$. Thus $W=Q_1 \dots Q_{m-1}$ is a colour $c$ tight walk in $\mathcal G^{c_i}\otimes F_i$. 

			Recall that, by \ref{itm:image-walk}, the image $\pi(W)$ of $W$ in $\HH$ is a monochromatic tight walk, of length at least $n/a_i \ge n/p$. We claim that each vertex appears at most $p$ times in $\pi(W)$. This would contradict the assumption on $\HH$, proving the claim.
			To show that no vertex appears too much in $\pi(W)$, notice that any vertex $u$ can appear in $\pi(Q_j)$ only for indices $j$ with $d_G(u, p_j)\leq c_i+hd_1$ (suppose $u\in \pi(Q_j)$. We have $\pi_0(P_j) = \pi_0(Q_j)$, since $P_j$ and $Q_j$ have the same root set. It follows that $p_j\in \pi_0(P_j)=\pi_0(Q_j)$, and $||Q_j||_{G}=||P_j||_{G}$. Using  Lemma~\ref{Lemma_distances_in_short_trees} gives $d_G(u, p_j)\leq ||Q_j||_{G} +hd_1=||P_j||_{G} + hd_1\leq c_i+hd_1$). 
			Therefore, every vertex is used at most $3r(\Delta+1)^{c_i+hd_1}\ll p$ times in $\pi(W)$, as claimed.  
		\end{proof}

		By the last claim, the first outcome does not hold, and so the second one does, meaning that we can cover at least $(1-200\eps)|U_{i}|$ vertices by vertex-disjoint grey $K_{d_{i+1}+1}$'s. Let $U_{i+1}$ be the set of vertices covered by by these cliques (which ensures that \ref{itm:forest-size-U} holds for $U_{i+1}$).
		For a $K_{d_{i+1}+1}$ with vertices $\{v_0, \dots, v_{d_{i+1}}\}$ define trees $F''_{i+1}(v_j)$ for $j \in [0,d_{i+1}]$, as follows.
		\begin{equation*}
			F''_{i+1}(v_j):=\left(v_j; F_i(v_0); \dots ; F_i(v_{j-1}); F_i(v_{j+1}); \dots; F_i(v_{d_{i+1}})\right)
		\end{equation*}
		(This is well defined due to $\pi(F_i(v_0)), \dots, \pi(F_i(v_{d_{i+1}}))$ being vertex-disjoint, which is true because \ref{itm:forest-separated} tells us that $F_i$ is $b_i$-separated in $G$, and because $d_G(v_s, v_t)\leq a_i\leq b_i$ for all $i$, as each $v_jv_{j'}$ is an edge in $G^{a_i}$.) Note that since $d_G(v_j, v_{j'})\leq a_i$ for all $j,j'$, level $1$ of each $F''_{i+1}(v_j)$ is $a_i$-short, and all other levels are $d_1$-short (from Lemma~\ref{Lemma_Augmentation_shortness} and \ref{itm:forest-separated}). 
		Define
		$F_{i+1}''=\bigcup_{v \in U_{i+1}}F''(v)$ (ordering so that $F''(u) < F''(v)$ if and only if $u < v$ in $V(G)$). Then $F_{i+1}''$ is a $V(G)$-forest with $\pi_0(F_{i+1}'') = U_{i+1}$.
		By Lemma~\ref{Lemma_augment_one_tree_properties} and Observation~\ref{Observation_rd_doesnt_change_on_subforests}, $\rd^{F_{i+1}''}|_{F_{i+1}''(v)^-}:F_{i+1}''(v)^-\to F_i$ is an monomorphism for all $v \in U_{i+1}$, i.e.\ $F_{i+1}''$ is an augmentation of $F_{i}$.

		Apply Lemma~\ref{Lemma_separate_trees} (with $(F,G,d,b,h,\Delta)_{\ref{Lemma_separate_trees}} = (F_{i+1}'',G,d_{i+1},b_{i+1},i+1,\Delta)$) to get a subforest $F_{i+1}'\leq F_{i+1}''$ of $b_{i+1}$-ary trees which is $b_{i+1}$-separated on $G$.
		Apply Lemma~\ref{Lemma_cleaning_forest} (with $(F,\HH,d,b,h,s,r,\Delta)_{\ref{Lemma_cleaning_forest}} = (F_{i+1}',\G^{c_{i+1}}, b_{i+1},c_{i+1},i+1,s,r,\binom{(\Delta+1)^{c_{i+1}}}{r})$) to get a $c_{i+1}$-ary subforest $F_{i+1} \le F_{i+1}'$ satisfying \ref{itm:forest-size-U} to \ref{itm:forest-clean} and \ref{itm:forest-augmentation}. 
		It remains to show \ref{itm:forest-disconnected}.
		\begin{claim}
			$\mathcal G^{c_{i+1}}\otimes F_{i+1}$   is $c_{i+1}$-disconnected on $(F_{i+1},G)$. 
		\end{claim}
		\begin{proof}
			Fix some $v \in v(G)$, as in the definition of ``$c_{i+1}$-disconnected''.
			Let $X,Y\subseteq L(F_{i+1}(v))$ be independent $(r-1)$-sets of the same type, and let $P$ be a length at most $3r$ monochromatic tight walk from $X$ to $Y$ in $\mathcal G^{c_{i+1}}\otimes F_{i+1}$. 
			Let $c$ be the colour of $P$ and $\tau$ be the type of $X$ and $Y$.
			By Lemma~\ref{Lemma_augmentation_properties}~\ref{itm:augment-e} (applied with $(F',F,G,s,t,k)_{\ref{Lemma_augmentation_properties}} = (F_{i+1},F_i,c_{i+1},c_i,a_i)$), we have that $\rd^{F_{i+1}}(P)$ is also a colour $c$ tight walk in $\mathcal G^{c_i}\otimes F_{i}$ from $\rd^{F_{i+1}}(X)$ to $\rd^{F_{i+1}}(Y)$, and $\rd^{F_{i+1}}(X)$ and $\rd^{F_{i+1}}(Y)$ are independent and are both of type $\tau^-$ (to see that $\rd^{F_{i+1}}(P)$ and $P$ have the same colour, recall that for any edge $e \in \mathcal G^{t}\otimes F$ its colour is the colour of $\pi(e)$. By \ref{Lemma_augmentation_properties}~\ref{itm:augment-a}, we have $\pi(\rd^{F_{i+1}}(e))=\pi(e)$ showing that $e$ and $\rd^{F_{i+1}}(e)$ have the same colour).  

			Let $\{v, u_1, \dots, u_{d_{i+1}}\}$ be the grey clique containing $v$, noting that children of $(v,v)$ in $F_{i+1}$ are contained in $\{(u_1,v), \dots, (u_{d_{i+1}},v)\}$. This shows that for $\ell\in L(F(v))$, we have $\pi_0(\rd^{F_{i+1}}(\ell))=\pi_1^{F_{i+1}}(\ell)\subseteq \{u_1, \dots, u_{d_{i+1}}\}$. 
			This tells us that $\pi_0(\rd^{F_{i+1}}(X)),\,\pi_0(\rd^{F_{i+1}}(Y))\subseteq \{u_1, \dots, u_{d_{i+1}}\}$, which, together with  Observation~\ref{Observation_indepndent_sets_have_single_root}, tells us that $\pi_0(\rd^{F_{i+1}}(X))=x$ and $\pi_0(\rd^{F_{i+1}}(Y))=y$ for some $x,y \in \{u_1, \dots, u_{d_{i+1}}\}$ (possibly with $x=y$), or, equivalently, $\rd^{F_{i+1}}(X)\subseteq F_{i+1}(x)$ and $\rd^{F_{i+1}}(Y)\subseteq F_{i+1}(y)$.

			We claim that $||\rd^{F_{i+1}}(P)||_G \ge c_{i}$. Indeed, if $x\neq y$, since the edge $xy$ is not coloured $(c,\tau)$ (it is coloured grey; in particular, it is indeed an edge in $G^{a_i}$), we have $||\rd^{F_{i+1}}(P)||_{G} \ge c_{i}$. 
			Otherwise, $x = y$ and then by $c_i$-disconnectedness of $F_{i}$, we again have $||\rd^{F_{i+1}}(P)||_{G}\geq c_i$. 
			By Lemma~\ref{Lemma_augmentation_properties} \ref{itm:augment-d}, we have $||P||_{G}> ||\rd^{F_{i+1}}(P)||_{G}-2a_{i}\geq c_i-2a_{i} \ge  c_{i+1}$, as required for showing $c_{i+1}$-disconnectedness.
		\end{proof}
		We have proved that $F_{i+1}$ satisfies \ref{itm:forest-size-U} to \ref{itm:forest-augmentation}, completing the proof of the theorem.
	\end{proof}

	We use $\mathcal H[t]$ to denote the $t$-blow-up of a hypergraph $\mathcal H$ --- the hypergraph formed by replacing each vertex $v$ by $t$ copies $v[1], \dots, v[t]$ and letting $E(\HH[t])$ be the set of $r$-sets $e$ of form $\{u_1[i_1], \dots, u_r[i_r]\}$ where $(u_1, \dots, u_r) \in E(\HH)$ and $i_1, \dots, i_r \in [t]$. Now we deduce the main theorem of the paper.
	\begin{proof}[Proof of Theorem~\ref{thm:main-path0}]
		Let $d\gg p,r,s$.
		Use Theorem~\ref{Theorem_walk} to find an $r$-uniform hypergraph $\mathcal H$ with $n$ vertices, maximum degree at most $p$, and such that ``in each $s$-colouring of $\mathcal H$ there is a tight walk of length at least $n/p$ in which every vertex is used at most $p$ times''. To prove \Cref{thm:main-path0}, it is sufficient to show that $\mathcal H[d]\sarrow \mathcal P_{n/p}$ (since $e(\mathcal H[d])\leq e(H) \binom{pr}{r} \le n \cdot p \binom{pr}{r}$). Consider an $s$-colouring of $\mathcal H[d]$. Let $S$ be a $V(\mathcal H)$-forest formed by assigning an (arbitrary) size $d$ star to each vertex, and define $\phi : \HH[d] \to \HH \otimes S$, by mapping each of $v[1], \dots, v[d]$ to a different leaf of the star $S(v)$, for every $v \in V(\HH)$. Notice that $\phi$ is an isomorphism, and colour $\HH \otimes S$ according to $\phi$, namely by colouring each edge $e$ by the colour of $\phi^{-1}(e)$.

		Lemma~\ref{Lemma_cleaning_forest} (applied with $(F,\HH,d,b,h,s,r,\Delta)_{\ref{Lemma_cleaning_forest}} = (S,\HH,d,p,1,s,r,p)$), gives us a $p$-ary subforest $S'\leq  S$ with $\mathcal H \otimes  S'$ cleanly coloured. By possibly relabelling the indices of the vertices $v[1], \dots, v[d]$ for $v \in V(\HH)$, we may assume that $\phi^{-1}(L(S'(v))) = \{v[1], \dots, v[p]\}$, for every $v \in V(\HH)$.
		Let $\mathcal H'$ be the subgraph of $\mathcal H[p]$ on $\{v[1]: v \in V(\HH)\}$, and notice that $\HH'$ is isomorphic to $\HH$. 
		The property of Theorem~\ref{Theorem_walk} thus gives us a monochromatic, say colour $c$, tight walk $\mathcal W$ in $\HH'$ of order at least $n/p$, that uses each vertex at most $p$ times. 
		Then $\W$ can be written as $w_1[1] \dots w_{\ell}[1]$, where $w_i \in V(\HH)$ and $\ell \ge n/p$. Let $m_i$ be the number of repetitions of the vertex $w_i$ in the subsequence $w_1, \dots, w_i$, noting that $m_i\leq p$ always. Now take $\mathcal W'=w_1[m_1] \dots w_{\ell}[m_{\ell}]$. 
		
		We claim that $\W'$ is a colour $c$ tight path of order $\ell$ in $\HH[p]$. Indeed, first notice that $w_i[m_i]$ is a vertex in $\HH[p]$ for every $i \in [\ell]$, as $m_i \le p$. Second, notice that the vertices $w_i[m_i]$ are distinct, by definition of $m_i$. Third, since $(w_i[1], \dots, w_{i+r-1}[1])$ is an edge in $\HH'$, we have that $\{w_i, \dots, w_{i+r-1}\}$ is an edge in $\HH$, showing that $(w_i[m_i], \dots, w_{i+r-1}[m_{i+r-1}])$ is an edge in $\HH[p]$. Finally, by cleanliness we have that the colours of $(w_i[1], \dots, w_{i+r-1}[1])$ and $(w_i[m_i], \dots, w_{i+r-1}[m_{i+r-1}])$ are the same, and thus by choice of $\W$ they are both $c$. Therefore, indeed, $\W'$ is a tight path of order $\ell \ge n/p$ in $\HH[p]$, proving the theorem. 
	\end{proof}

\section{Conclusion} \label{sec:conc}

    In this paper (and its previous version \cite{letzter2021size}) we studied families of bounded degree hypergraphs whose size-Ramsey number is linear in their order. Not much is known about bounded degree hypergraphs, or even graphs, whose size-Ramsey numbers are not linear. To this end, we would like to reiterate the conjecture of R\"odl and Szemer\'edi~\cite{rodl2000size} mentioned in the introduction, as perhaps the most interesting question in this area. 
    
    \begin{conj}
        For every $d\geq 3$ there exist $\eps > 0$ and a sequence of graphs $(G_n)$ on $n$ vertices and maximum degree at most $d$ such that $\rhat(G_n) = \Omega(n^{1+\eps})$.
    \end{conj}

    We should mention that this conjecture makes sense also in the hypergraph setting, and, in fact, the original construction from~\cite{rodl2000size} can be generalised to the hypergraph setting, as shown in~\cite{dudek2017size}.
    
    Regarding upper bounds, the only known examples of bounded degree graphs with superlinear size-Ramsey number are the R\"odl--Szemer\'edi construction \cite{rodl2000size} and its modification by Tikhomirov \cite{tikhomirov2024}. One candidate that seems worth considering is the grid graph. The $n\times n$ \emph{grid} graph $G_{n,n}$ is defined to have vertex set $[n]\times [n]$ with $uv$ being an edge whenever $u$ and $v$ differ in exactly one coordinate and the difference is exactly $1$. The best current upper bound on the size-Ramsey number of the $n \times n$ grid is of order $O(n^{5/2})$ by Conlon, Nenadov and Truji\'c \cite{CNT2023}, improving on an earlier upper bound of order $n^{3+o(1)}$ by Clemens, Miralaei, Reding, Schacht and Taraz~\cite{clemens2021size}. No lower bounds are known except the trivial $\Omega(n^2)$. It would thus be very interesting to have an answer to the following question.
    
    \begin{question}
        Is the size-Ramsey number of the $n \times n$ grid graph vertices $O(n^2)$?
    \end{question}
    
Turning to hypergraphs, Dudek, La Fleur, Mubayi and R\"odl \cite{dudek2017size} asked for the maximum size-Ramsey number of an $r$-uniform $\ell$-tree (with possibly unbounded maximum degree). Given $r\geq 3$ and $1\leq \ell \leq r-1$, an \emph{$r$-uniform $\ell$-tree} is an $r$-graph with edges $\{e_1, \ldots, e_m\}$ such that for every $i \in \{2, \ldots, m\}$ we have $\big|e_i \cap (\bigcup_{1 \le j < i} e_j)\big| \le \ell$ and $e_i \cap (\bigcup_{1 \le i < j} e_j) \subseteq e_{i_0}$ for some $i_0 \in [i-1]$ (in particular, an $r$-uniform tight path is an $r$-uniform $(r-1)$-tree).  The authors in~\cite{dudek2017size} showed that if $\T$  is an $r$-uniform $\ell$-tree then $\rhat(\T)=O(n^{\ell+1})$, which is tight when $\ell=1$, and asked whether the bound is tight for all $\ell$. We suspect the answer is positive, and reiterate the question here.
\begin{question} 
    For any  $r\geq 3$ and $1\leq \ell \leq r-1$, is it true that for every $n$ there exists an $r$-uniform 
    $\ell$-tree $\T$ of order at most $n$ such that $\rhat(\T)=\Omega(n^{\ell+1})$?
\end{question}

Another interesting problem is to study the tightness of known bounds on the size-Ramsey number of $q$-subdivisions of bounded degree graphs, where $q$ is fixed. In an aforementioned paper \cite{draganic2021rolling}, Dragani\'c, Krivelevich and Nenadov recently showed that for fixed $d$ and $q$, the size-Ramsey number of the $q$-subdivision of an $n$-vertex graph with maximum degree $d$ is bounded by $O(n^{1+1/q})$. This is tight if the host graph is a random graph but just as in the grid case, we suspect it might not be tight in general. We now pose our final question.
\begin{question}
    For fixed $d$ and $q$, is there a sequence $(G_n)$ where $G_n$ is an $n$-vertex graph with maximum degree at most $d$ such that the size-Ramsey number of the $q$-subdivision of $G_n$ is at least $\Omega(n^{1+1/q})$? 
\end{question}

    
    
	\bibliography{size-ramsey}
	\bibliographystyle{amsplain}
 
\end{document}